\RequirePackage[table]{xcolor}
\documentclass[a4paper,12pt]{amsart}
\usepackage{amsfonts}
\usepackage{amsmath}
\usepackage{ifthen}
\usepackage{amsrefs}
\usepackage{amsthm}
\usepackage[all,cmtip]{xy}
\usepackage{amssymb}
\usepackage{hyperref}
\usepackage[tocflat]{tocstyle}
\usetocstyle{standard}
\usepackage{verbatim}
\usepackage{hyperref}
\usepackage{url}
\usepackage[shortlabels]{enumitem}
\usepackage[paper=a4paper,left=20mm,right=20mm,top=25mm,bottom=30mm]{geometry}

\usepackage{tikz}
\usetikzlibrary{backgrounds}
\usepackage{tkz-euclide}
\usepackage{xcolor}
\usetikzlibrary{trees,snakes,shapes.geometric}
\usepackage[outline]{contour}
\contourlength{1.5pt}
\usetikzlibrary{positioning}
\usetikzlibrary{%
  matrix,%
  calc,%
  arrows%
}

\setlist[enumerate]{topsep=0em, itemsep= -0em, parsep = 0 em, label=$(\alph*)$}

\let\emptyset\varnothing
\nocite{*}

\newcommand{\NN}{\mathbb{N}}

\newcommand{\ZZ}{\mathbb{Z}}
\newcommand{\RR}{\mathbb{R}}
\newcommand{\PP}{\mathbb{P}}

\newcommand{\cD}{\mathcal{D}}

\newcommand{\cE}{\mathcal{E}}

\newcommand{\cQ}{\mathcal{Q}}

\newcommand{\cX}{\mathcal{X}}

\DeclareMathOperator{\MaxRk}{MaxRk}

\DeclareMathOperator{\fF}{\mathfrak{F}}

\DeclareMathOperator{\Rad}{Rad}

\DeclareMathOperator{\rep}{rep}

\DeclareMathOperator{\Char}{char}

\DeclareMathOperator{\Hom}{Hom}
\DeclareMathOperator{\coker}{coker}
\DeclareMathOperator{\rk}{rk}

\DeclareMathOperator{\re}{re}
\DeclareMathOperator{\modd}{mod}
\DeclareMathOperator{\supp}{supp}

\DeclareMathOperator{\Ext}{Ext}

\DeclareMathOperator{\id}{id}
\DeclareMathOperator{\im}{im}
\DeclareMathOperator{\End}{End}

\DeclareMathOperator{\EIP}{EIP}

\DeclareMathOperator{\Inj}{Inj}

\DeclareMathOperator{\Ad}{Ad}
\DeclareMathOperator{\EKP}{EKP}

\DeclareMathOperator{\IJT}{IJT}
\DeclareMathOperator{\CJT}{CJT}

\DeclareMathOperator{\dimu}{\underline{dim}}

\let\emptyset\varnothing
\newtheorem{proposition}{Proposition}[section]
\newtheorem{Theorem}[proposition]{Theorem}
\newtheorem{Lemma}[proposition]{Lemma}

\newtheorem{corollary}[proposition]{Corollary}

\newtheorem{Example}[proposition]{Example}
\newtheorem*{TheoremN}{Theorem}

\newtheorem*{corollaryN}{Corollary}
\newtheorem*{PropositionN}{Proposition}

\newenvironment{example}[1][Example.]{\begin{trivlist}
\item[\hskip \labelsep {\bfseries #1}]}{\end{trivlist}}

\newenvironment{Remark}[1][Remark.]{\begin{trivlist}
\item[\hskip \labelsep {\bfseries #1}]}{\end{trivlist}}

\newenvironment{Definition}[1][Definition.]{\begin{trivlist}
\item[\hskip \labelsep {\bfseries #1}]}{\end{trivlist}}

\begin{document}

\rmfamily


\thispagestyle{empty}
\author{Daniel Bissinger}

\title{Indecomposable Jordan types of Loewy length $2$} 
\address{Christian-Albrechts-Universit\"at zu Kiel, Ludewig-Meyn-Str. 4, 24098 Kiel, Germany}
\email{bissinger@math.uni-kiel.de}

\maketitle

\makeatletter
\def\blfootnote{\gdef\@thefnmark{}\@footnotetext}
\makeatother

\blfootnote{\textup{2010} \textit{Mathematics Subject Classification}: 16G20}


\begin{abstract}
Let $k$ be an algebraically closed field, $\Char(k) = p \geq 2$ and $E_r$ be a $p$-elementary abelian group of rank $r \geq 2$. Let $(c,d) \in \NN^2$. 
We show that there exists an indecomposable module of constant Jordan type $[1]^c [2]^d$ and Loewy length $2$ if and only if $q_{\Gamma_r}(d,d+c) \leq 1$ and $c \geq r-1$, where $q_{\Gamma_r}(x,y) := x^2 + y^2-rxy$ denotes the Tits form of the generalized Kronecker quiver $\Gamma_r$.\\
Since $p > 2$ and constant Jordan type $[1]^c [2]^d$ imply Loewy length $\leq 2$, we get in this case the full classification of Jordan types $[1]^c [2]^d$ that arise from indecomposable modules.
\end{abstract}

\section*{Introduction}

Let $r \geq 2$, $k$ be an algebraically closed field of characteristic $p > 0$ and $E_r$ be a $p$-elementary abelian group of rank $r$. It is well known that the category of finite-dimensional $kE_r$-modules $\modd kE_r$ is of wild type, whenever $p \geq 3$ or $p = 2$ and $r > 2$. Therefore subclasses with more restrictive properties have been studied; in  \cite{CFP1}, the subclass of modules of constant Jordan type and modules with even more restrictive properties, called equal images property and equal kernels property, were introduced.\\
But even these smaller subcategories have turned out to be wild $($see \cite[5.5.5]{Be1} and \cite{Bi1}$)$  and the classification of their objects therefore is considered hopeless. On the other hand, since such modules give rise to vector bundles $($see \cite[8.4.11]{Be1}$)$, the presence of many indecomposables may lead to the construction of interesting bundles.

\noindent Based on these results and the considerations in \cite{Be1} this work is concerned with the constant Jordan types that arise from $kE_r$-modules of Loewy length $2$. If we allow arbitrary modules of constant Jordan type, a complete answer is given in \cite[10.5.1]{Be1} $($the proof given for $p=2$ works in general$)$:

\begin{PropositionN} Let $\Char(k) = p > 2$. There exists a module of Loewy length $2$ and constant Jordan type $[1]^c [2]^d$ in $\modd kE_r$ if and only if $(c,d) \in  \NN_{\geq r-1} \times \NN$.
\end{PropositionN}

The modules constructed in the proof of the result above are far from being indecomposable. In fact, such a module has at least $c-(r-1) + 1 = c - r + 2$ direct summands.
We study  constant Jordan types that arise from indecomposable $kE_r$-modules with the equal images or the equal kernels property of Loewy length $2$. Since modules of Loewy length $2$ are closely related to representations of the Kronecker quiver $\Gamma_r$ with $r$ arrows, we study this problem in the hereditary category $\rep(\Gamma_r)$ of finite dimensional representations of $\Gamma_r$. We denote by $q_{\Gamma_r} \colon \NN^2_0 \to \ZZ, (x,y) \mapsto x^2+y^2- rxy$ the Tits form of $\Gamma_r$. \\
Using recent results $($see \cite{Ri10}$)$ on elementary representations of $\Gamma_r$ for $r \geq 3$, we show that the generic Jordan type $[1]^{c_M} [2]^{d_M}$ of an indecomposable, non-simple representation $M \in \rep(\Gamma_r)$ is contained in 
\[ \IJT := \{ (c,d) \in \NN^2 \mid q_{\Gamma_r}(d,d+c) \leq 1,c \geq r - 1 \}.\]
Then we show the existence of an indecomposable representation $M$ $($for abitrary characteristic$)$ in $\rep(\Gamma_r)$ that has the equal kernels property and constant Jordan type $[1]^c [2]^d$ for each $(c,d) \in \IJT$. We arrive at this result by considering the universal covering $\pi \colon C_r \to \Gamma_r$ of the Kronecker quiver in conjunction with Kac's Theorem and a homological characterization of the representations with the equal kernels property in $\rep(\Gamma_r)$. 
In the end we transport our results back to $\modd kE_r$ and conclude:

\begin{TheoremN}
Let $\Char(k) = p > 0$, $r \geq 2$ and $(c,d) \in \NN^2_0$.
 The following statements are equivalent:\begin{enumerate}
\item  There exists an indecomposable $kE_r$-module of constant Jordan type $[1]^c [2]^d$ and Loewy length $2$.
\item There exists an indecomposable $kE_r$-module with the equal images property of constant Jordan type $[1]^c [2]^d$ and Loewy length $2$.
\item There exists an indecomposable $kE_r$-module with the equal kernels property of constant Jordan type $[1]^c [2]^d$ and Loewy length $2$.
\item $(c,d) \in \IJT$.
\end{enumerate}
\end{TheoremN}

\noindent If in addition $\Char(k) = p > 2$, we get:
\begin{corollaryN}
Let $\Char(k) = p > 2$, $r \geq 2$. For each element $(c,d) \in \NN_0 \times \NN$ the following statements are equivalent:\begin{enumerate}
\item  There exists an indecomposable $kE_r$-module of constant Jordan type $[1]^c [2]^d$.
\item There exists an indecomposable $kE_r$-module with the equal images property of constant Jordan type $[1]^c [2]^d$.
\item There exists an indecomposable $kE_r$-module with the equal kernels property of constant Jordan type $[1]^c [2]^d$.
\item $(c,d) \in \IJT$.
\end{enumerate}
\end{corollaryN}

\noindent As another consequence of our considerations in $\rep(\Gamma_r)$ we obtain a refinement of \cite[5.5.5]{Be1}:

\begin{TheoremN}
Let $\Char(k) = p > 0$, $r \geq 3$ and $(c,d) \in \NN^2_0$ such that $r-1 \leq c \leq (r-1)d$. Then the full subcategory of modules with constant Jordan type $[1]^{nc} [2]^{nd}$, $n \in \NN$ has wild representation type.
\end{TheoremN}

\section{Preliminaries}
\label{Preliminaries}

\subsection{Module properties}

We let $k$ be an algebraically closed field and $r \in \NN_{\geq 2}$. We denote by $\Gamma_r$ the $r$-Kronecker quiver with two vertices $1,2$ and $r$ arrows $\gamma_1,\ldots,\gamma_r \colon 1 \to 2$.\\
Let $\Char(k) = p > 0$ and denote by $E_r$ a $p$-elementary abelian group of rank $r$. Choose generators $g_1,\ldots,g_r$ and define $x_i := g_i - 1$ for all $1 \leq i \leq r$. We get an isomorphism $kE_r \cong k[X_1,\ldots,X_r]/(X^p_1,\ldots,X^p_r)$ of $k$-algebras by sending $x_i$ to $X_i + (X^p_1,\ldots,X^p_r)$ for all $i \in \{1,\ldots,r\}$. We get a functor $\fF \colon \rep(\Gamma_r) \to \modd kE_r$ by assigning to each representation $M \in \rep(\Gamma_r)$ the $kE_r$-module $\fF(M)$ with underlying vector space $M_1 \oplus M_2$ and $x_i . (m_1 + m_2) := M(\gamma_i)(m_1)$. Morphisms are defined in the obvious way. Although the functor sends the two simple objects $I_1,P_1$ of $\rep(\Gamma_r)$ to the uniquely determined simple module $k$ of $kE_r$, the functor has very nice properties:

\begin{proposition}\cite[5.1.2]{Far2}\label{FunctorESP}
Let $M \in \modd kE_r$. 
\begin{enumerate}
\item If $M$ has Loewy length $\leq 2$, then there is a representation $N \in \rep(\Gamma_r)$ such that $M \cong \fF(N)$.
\item $N \in \rep(\Gamma_r)$ is indecomposable if and only if $\fF(N)$ is indecomposable.	
\end{enumerate}
\end{proposition}

\noindent We say that $M \in \modd kE_r$ has \textsf{constant Jordan Type}   if for all $\alpha \in k^r \setminus \{0\}$ the Jordan type of the $($nilpotent$)$ operator 
$x^M_\alpha \colon M \to M, m \mapsto x_\alpha.m = (\sum^r_{i=1} \alpha_i x_i).m$ is independent of $\alpha$. The module has the \textsf{equal images property} $($\textsf{equal kernels property}$)$ if the image $($resp. kernel$)$ of $x^M_\alpha$ does not depend on $\alpha$. For the Kronecker quiver we have the following natural analogous definitions, that do not require $\Char(k) > 0$. 

\begin{Definition} 
Let $M \in \rep(\Gamma_r)$ be a representation. 
\begin{enumerate}[topsep=0em, itemsep= -0em, parsep = 0 em, label=$(\alph*)$]	
\item $M$ has \textsf{constant rank} if the rank of $M^\alpha := \sum^r_{i=1} \alpha_i M(\gamma_i)$ is independent of $\alpha \in k^r \setminus \{0\}$.
\item $M$ has the \textsf{equal images property} $($resp. \textsf{equal kernels property}$)$ if $\im M^\alpha$ $($resp. $\ker M^\alpha)$ is independent of $\alpha \in k^r \setminus \{0\}$.
\end{enumerate}
\end{Definition}

\noindent Observe that we can identify $M^\alpha \colon M_1 \to M_2$ with the nilpotent operator of degree $\leq 2$
\[x^{\fF(M)}_\alpha \colon \fF(M) \to \fF(M), m_1 \oplus m_2 \mapsto x_\alpha.(m_1 + m_2).\]

 The rank of $x^{\fF(M)}_\alpha$ does not depend on $\alpha$ if and only if $M$ has constant rank. Since the Jordan canonical form of $x^{\fF(M)}_\alpha$ has exactly $\rk(M^\alpha)$ blocks of size $2$ and $\dim_k M_1 + \dim_k M_2 - 2\rk(M^\alpha)$ blocks of size $1$, we see that the Jordan type of $x^{\fF(M)}_\alpha$ is independent of $\alpha$ if and only if $M^\alpha$ has constant rank.

 \begin{Definition}
 Let $M \in \rep(\Gamma_r)$, then $M$ has \textsf{constant Jordan type}, provided $\rk(M^\alpha)$ does not depend on $\alpha \in k^r \setminus \{0\}$.  We say that $M$ has \textsf{constant Jordan type} $[1]^c [2]^d$, provided $M$ has constant Jordan type and the Jordan canoncial form has exactly $c \in \NN_0$ blocks of size $1$ and exactly $d \in \NN_0$ blocks of size $2$.
  \end{Definition}
  
\noindent  Clearly, if $M$ has the equal kernels or the equal images property, then $M$ has constant Jordan type. 
It is also easy to see that an indecomposable and non-simple representation $M \in \rep(\Gamma_r)$ has the equal images property $($resp. equal kernels property$)$ if and only if $\im M^\alpha =M_2$ $($resp. $\ker M^\alpha = \{0\})$  for all $\alpha \in k^r \setminus \{0\}$. Therefore the following definitions make sense. 

\begin{Definition}
We define  
\[ \EKP := \{ M \in \rep(\Gamma_r) \mid \forall \alpha \in k^r \setminus \{0\} \colon \ker M^\alpha = \{0\} \}, \] 
\[ \EIP := \{ M \in \rep(\Gamma_r) \mid \forall \alpha \in k^r \setminus \{0\} \colon \im M^\alpha = M_2 \},\]
\[ \CJT := \{ M \in \rep(\Gamma_r) \mid \exists d \in \NN_0 \forall \alpha \in k^r \setminus \{0\} \colon \rk M^\alpha = d \}.\]  
\end{Definition}

\noindent We denote by $\modd_{2} kE_r$ the $kE_r$-modules of Loewy length $\leq 2$. In view of the following result, it is wellfounded to study modules of constant Jordan type in the hereditary category $\rep(\Gamma_r$).

\begin{proposition}\cite[2.1.2]{Wor1}\label{Proposition:WorchEmbedding}
Let $\Char(k) = p > 0$. For $\cX \in \{ \EIP,\EKP,\CJT\}$ the restriction of $\fF$ to $\cX$ induces a faithful exact functor $\fF_{\cX} \colon \cX \to \modd_2 kE_r$
such that 
\begin{enumerate}
\item for $\cX = \EIP$, $\fF_X$ reflects isomorphisms and the essential image consists of the modules in $\modd_2 kE_r$ that have the equal images property.
\item for $\cX = \EKP$, $\fF_X$ reflects isomorphisms and the essential image consists of the modules in $\modd_2 kE_r$ that have the equal kernels property.
\item for $\cX = \CJT$, the essential image consists of the modules in $\modd_2 kE_r$ that have constant Jordan type.
\end{enumerate}
\end{proposition}

\noindent The following result provides a functorial characterization of the forementionend categories.

\begin{Theorem}\cite[2.2.1]	{Wor1}\label{Theorem:Worch}
Let $r \geq 2$. There exists a family of indecomposable representations $(X_\alpha)_{\alpha \in k^r \setminus \{0\}}$, such that the following statements hold:
\begin{enumerate}
\item $\EKP = \{ M \in \rep(\Gamma_r) \mid \forall \alpha \in k^r \setminus \{0\} \colon \Hom(X_\alpha,M) = 0 \}$.
\item $\EIP = \{ M \in \rep(\Gamma_r) \mid \forall \alpha \in k^r \setminus \{0\} \colon \Ext^1(X_\alpha,M) = 0 \}$.
\item $\CJT = \{ M \in \rep(\Gamma_r) \mid \exists c \in \NN_0 \forall \alpha \in k^r \setminus \{0\} \colon \dim_k \Hom(X_\alpha,M) = c \}$.
\end{enumerate}
\end{Theorem}

\subsection{The Auslander-Reiten quiver of $\Gamma_r$} \label{Section:ARquiver}

Recall that for $r \geq 2$ there are infinitely many isomorphism classes of indecomposable representations of $\Gamma_r$. We denote by $\tau_{\Gamma_r}$ the Auslander-Reiten translation of $\Gamma_r$. The indecomposable representations fall into three classes: an indecomposable representation $M$ is called \textsf{preprojective} $($\textsf{preinjective}$)$ if and only if $M$ is in the $\tau_{\Gamma_r}$-orbit of a projective $($injective$)$ indecomposable representation.  All other indecomposable representations are called \textsf{regular}. We call a representation $M \in \rep(\Gamma_r)$ \textsf{preprojective} $($\textsf{preinjective}, \textsf{regular}$)$ if all indecomposable direct summands of $M$ are preprojective $($resp. preinjective, regular$)$.
There are up to isomorphism two indecomposable projective representations $P_1,P_2$, two injective representations $I_1,I_2$ and two simple representations $I_1,P_2$.
  We define $P_{i+2} := \tau^{-1}_{\Gamma_r} P_i$ and $I_{i+2} := \tau_{\Gamma_r} I_i$ for all $i \in \NN$.
 The Auslander-Reiten quiver of $\Gamma_r$  quiver looks as follows:
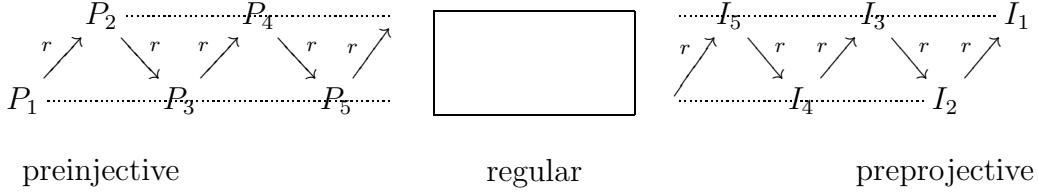
\[ 
\xymatrix @C=10pt@R=15pt{
& P_2 \ar^{r}[dr] \ar@{..}[rrrr]&  & P_4 \ar^{r}[dr] &   & &   & &  &  &  & \ar@{..}[rrrrr] & I_5 \ar^{r}[dr]  & & I_3 \ar^{r}[dr]& &  I_1\\
P_1   \ar@{..}[rrrrr] \ar^{r}[ur]&  & P_3 \ar^{r}[ur] & & P_5 \ar^{r}[ur]  & &  & & & & & \ar^{r}[ur] \ar@{..}[rrrr]  & & I_4 \ar^{r}[ur] & & I_2 \ar^{r}[ur] \\
& & & & & & & & & & & & & & & &
\save "1,7"."2,11"*[F]\frm{} \restore  
\save "3,2"."3,7" *\txt{preinjective} \restore
\save "3,9"."3,10" *\txt{regular} \restore
\save "3,16"."3,17" *\txt{preprojective} \restore}
\] 
The arrows in the components that contain all preinjective and preprojective indecomposable representations are $r$-folded. They have dimension vectors $\dimu P_1 = (0,1)$, $\dimu P_2 = (1,r)$, $\dimu I_1 = (1,0)$ and $\dimu I_2 = (r,1)$. Moreover, we have $\dimu X_{i+2} = r \dimu X_{i+1} - \dimu X_i$ for all $i \in \NN$ and $X \in \{P,I\}$. We conclude with induction:

\begin{Lemma}\label{Lemma:PropertiesARquiver}The following statements hold.
\begin{enumerate} 
\item We have $(\dimu P_i)_1 < (\dimu P_i)_2$ and $(\dimu I_i)_1 > (\dimu I_i)_2$ for all $i \in \NN$.
\item Each non-zero preprojective representation $X$ satisfies $\dim_k X_1 < \dim_k X_2$.
\item Each non-zero preinjective representation $Y$ satisfies $\dim_k Y_1 > \dim_k Y_2$.
\item We have $(\dimu I_i)_j  < (\dimu I_{i+1})_{j}$ for all $i \in \NN$ and $j \in \{1,2\}$. 
\end{enumerate}
\end{Lemma}

\subsection{Kac's Theorem}

\label{Section:Kac's Theorem}

\noindent In this section we recall a theorem of Kac and prove the existence of certain roots that will be needed later on. The field $k$ is of abitrary characteristic. For a more detailed description we refer the reader to \cite{Kac1}, \cite{Kac2} and \cite{CB1}.

 Let $Q$ be an acyclic quiver without loops with finite vertex set $Q_0 = \{1,\ldots,n\}$. For $x \in Q_0$ we define $x^+_Q := \{ y \in Q_0 \mid \exists x \to y\}$, $x^-_Q := \{ y \in Q_0 \mid \exists y \to x \}$ and $n_Q(x) := x^+ \cup x^-$. The quiver $Q$ defines a $($non-symmetric$)$ bilinear form $\langle \ , \ \rangle_Q \colon \ZZ^n \times \ZZ^n \to \ZZ$, given by 
\[ (x,y) \mapsto \sum^n_{i=1} x_i y_i - \sum_{i \to j \in Q_1} x_i y_j,\]
which coincides with the \textsf{Euler-Ringel form} on the Grothendieck group of $Q$, i.e. for $X,Y \in \rep(Q)$ we have 
\[ \langle \dimu X, \dimu Y \rangle_Q = \dim_k \Hom_Q(X,Y) - \dim_k \Ext^1_Q(X,Y).\]
The \textsf{Tits form} is defined by $q_Q(x) := \langle x,x \rangle$. We denote the \textsf{symmetric form} corresponding to $\langle \ , \ \rangle_Q$ by $( \ , \ )_{Q}$, i.e. $(x,y)_Q := \langle x,y \rangle_Q + \langle y,x \rangle_Q$.

\noindent For each $i \in Q_0$ we have an associated reflection $r_i \colon \ZZ^n \to \ZZ^n$ given by $r_i(x) := x - e_i (x,e_i)_Q$, where $e_i \in \ZZ^n$ denotes the $i$-th canonical basis vector. By definition we have
\[ r_i(x)_j = \begin{cases}
x_j,\text{for} \ j \neq i \\ 
-x_i + \sum_{l \to i, i \to l} x_l,\text{for} \ j = i.
\end{cases}
\]
We denote by $W_Q := \langle r_i \mid i \in \{1,\ldots,n\} \rangle$ the \textsf{Weyl group} associated to $Q$ and by $\Pi_Q := \{e_1,\ldots,e_n\}$  the set of \textsf{simple roots}. The set
\[F_Q := \{ \alpha \in \NN^n_0 \setminus \{0\} \mid \forall i \in \{1,\ldots,n\} \colon (\alpha,e_i)_Q \leq 0, \supp(\alpha) \ \text{is connected}\}\] is called the \textsf{fundamental domain} of the Weyl group action.

\begin{Definition}
We define 
\[ \Delta_+(Q) = \Delta^{\re}_+(Q)\sqcup \Delta^{\im}_+(Q),\]
where $\Delta^{\re}_+(Q) := W_Q \Pi_Q \cap \NN^n_0$ and $\Delta^{\im}_+(Q) := W_Q F_Q$.
The elements in $\Delta_+(Q)$ are called \textsf{(positive) roots} of $q_Q$.
\end{Definition}

\noindent We formulate a simplified version of Kac's Theorem that suffices for our purposes.

\begin{Theorem}[Kac's Theorem]  \label{Theorem:Kac}\cite[Theorem B]{Kac3}, \cite[Theorem \S 1.10]{Kac2}
Let $k$ be an algebraically closed field and $Q$ an acyclic finite, connected quiver without loops and vertex set $\{1,\ldots,n\}$. Let $\alpha \in \NN^n_0$.

\begin{enumerate}
\item There exists an indecomposable representation in $\rep(Q)$ with dimension vector $\alpha$ if and only if $\alpha \in \Delta_+(Q)$.
\item If $\alpha \in \Delta^{\re}_+(Q)$, then there exists a unique	 indecomposable representation $M_\alpha \in \rep(Q)$ with $\dimu M = \alpha$.
\end{enumerate} 
\end{Theorem}

\begin{Example}\label{Examples} \normalfont
We consider the Kronecker quiver $\Gamma_r$ with $r$ arrows. For $x  \in \ZZ^{2}$ we have 
\[q_{\Gamma_r}((x_1,x_2)) = x_1^2 + x_2^2 -rx_1x_2,\] hence
\[ F_{\Gamma_r} = \{(a,b) \in \NN^2 \mid \frac{2}{r} \leq \frac{b}{a}  \leq \frac{r}{2} \}.\] 
Moreover one can show $($\cite[Section 2.6]{Kac1},\cite[2]{BoChen1}$)$ that $\Delta^{\re}_+(\Gamma_r) = \{ (a,b) \in \NN^2_0 \mid q_{\Gamma_r}(a,b) = 1 \}$ and  $\Delta^{\im}_+(\Gamma_r) = \{ (a,b) \in \NN^2 \mid q_{\Gamma_r}(a,b) \leq 0 \}$. Hence direct computation shows
\[ \boxed{
\Delta^{\im}_+(\Gamma_r)  = \{ (a,b) \in \NN^2 \mid q_{\Gamma_r}(a,b) \leq 0 \}   =  \begin{cases} \{ (a,a) \mid a \in \NN \}, r = 2 \\ 
  \{(a,b) \in \NN^2 \mid (\frac{r-\sqrt{r^2-4}}{2}) < \frac{b}{a} <  (\frac{r+\sqrt{r^2-4}}{2})\}, r \geq 3.
  \end{cases} }
 \]
 We also have for $M \in \rep(\Gamma_r)$ indecomposable the equivalence $($see \cite[2]{BoChen1}$)$
 \[ \boxed{M \ \text{regular} \Leftrightarrow q_{\Gamma_r}(\dim_k M_1,\dim_k M_2) < 1.}\]
\noindent We will use these results often later on.
 
\end{Example}

\subsection{Subquivers of $C_r$}

 Consider the universal cover $C_r$ of the quiver $\Gamma_r$. The underlying graph of $C_r$ is an $($infinite$)$ $r$-regular tree and $C_r$ has bipartite orientation. That means each vertex $x \in (C_r)_0$ is a sink or a source and $|n_{C_r}(x)| = r$.

Let $a \in \NN_{\geq 1}$ and $Q(a) \subseteq C_r$ be a connected subquiver with $a$ sources such that $n_{C_r}(x) \subseteq Q(a)_0$ for each source $x \in Q(a)_0$. It is easy to see that $Q(a)_0$ contains exactly $b := a(r-1)+1$ sinks. We call such a quiver \textsf{source-regular} with $a$ sources. Note that two source-regular quivers with $a$ sources are in general not isomorphic if $a \geq 3$.
We label the sources of $Q(a)_0$ by $1,\ldots,a$ and define $n_a(x) := n_{Q(a)}(x)$ for all $x$ in $Q(a)_0$. Recall that a vertex $x \in Q(a)_0$ is called \textsf{leaf} if $|n_a(x)| \leq 1$.

\vspace{0.5cm}

\noindent The following results will be needed later.

\begin{Lemma}\label{Lemma:ImaginaryRoot}
Let $a \geq 1$, $Q(a)$ be source-regular with $a$ sources, $\alpha \in \NN_0^{{Q(a)}_0}$ such that $\alpha_i = 1$ for all $i \in \{1,\ldots,a\}$, $\supp(\alpha) = {Q(a)}_0$ and for each sink $l$ we have $\alpha_l \leq \max\{1,|n_a(l)| - 1\}$. Then $\alpha \in \Delta_+(Q(a))$.
\end{Lemma}
\begin{proof}
Observe that by assumption we have $\alpha_l = 1$ for each leaf $l \in Q(a)_0$.
We prove the statement by induction on $a \in \NN$. For $a = 1$, the vertex set of $Q(a)$ consists exactly of one source $a$ and its neighbourhood $\{y_1,\ldots,y_r\}$. By assumption $\alpha_i = 1$ for all $i \in Q(a)_0$. We set $\beta := r_{y_r} \circ \cdots \circ r_{y_1}(\alpha)$ and get that $\beta_z = \delta_{z a}$. Hence $\beta \in \Pi_{Q(a)}$ and $\alpha \in \Delta_+(Q(a))$. \\
Now let $a > 1$. Since $Q(a)$ is a tree and $a > 1$, we find a source $x \in Q(a)_0$ such that $n_a(x) = \{y_1,\ldots,y_{r-1},y\}$, $y_1,\ldots
,y_{r-1}$ are leaves and $|n_a(y)| > 1$.
We assume without loss of generality that $x = a$. We let $Q^\prime$ be the full subquiver of $Q(a)$ with vertex set $Q(a)_0 \setminus \{a,y_1,\ldots,y_{r-1}\}$. The quiver $Q^\prime$ is a tree and source-regular with $a-1$ sources. We distinguish two cases:\\
If $\alpha_y = |n_a(y)| - 1$, we set 
$\beta := r_{y_{r-1}} \circ \cdots \circ r_{y_1}(\alpha)$. Then $\beta$ satisfies $\beta_z = 0$ if $z \in \{y_1,\ldots,y_{r-1}\}$ and $\beta_z = \alpha_z$ otherwise. Hence $r_y(\beta)_y = -\beta_y + \sum_{z \to y} \beta_z = -\alpha_y + 1 \cdot |n_a(y)| = 1$, which implies that 
\[ (r_a r_y(\beta))_z = \begin{cases}
\alpha_z, \ \text{for} \ z \not\in \{a,y_1,\ldots,y_{r-1},y\} \\
1, \ \text{for} \ z = y\\
0, \ \text{else}.
\end{cases}
\]
Hence $\alpha^\prime := r_a r_y(\beta)$ satisfies $\supp(\alpha^\prime) = Q^\prime_0$ and  $\alpha^\prime_l \leq \max \{1, |n_{Q^\prime}(l)|-1\}$ for each sink $l \in Q^\prime_0$. The inductive assumption implies that $\alpha^\prime \in \Delta_+(Q^\prime)$. By $\ref{Theorem:Kac}$ there is an indecomposable representation $M^\prime$ of $Q^\prime$ with dimension vector $\alpha^\prime$. Since $M^\prime$ is an indecomposable representation for $Q(a)$, we conclude with $\ref{Theorem:Kac}$ that $\alpha^\prime \in \Delta_+(Q(a))$. Hence $\alpha \in \Delta_+(Q(a))$.\\
Now assume that $\alpha_y \leq |n_a(y)| - 2$. We consider $\beta \in \NN^{Q^\prime_0}_0$ with $\beta_i := \alpha_i$ for all $i \in Q^\prime_0$. We have then $\beta_y \leq |n_{Q^\prime}(y)| - 1$. By assumption $\beta \in \Delta_+(Q^\prime)$ and Kac's Theorem implies the existence of an indecomposable representation $M \in \rep(Q^\prime)$ with $\dimu M = \beta$. We define an indecomposable representation $N \in \rep(Q(a))$ by setting $N_{\mid Q^\prime} := M$, $N_z := k$ for all $z \in \{y_1,\ldots,y_{r-1},a\}$, $N(a \to y_{i}) := \id_k$ for all $1 \leq i \leq r-1$ and we let $N(a \to y) \colon k \to M_y$ be an injective $k$-linear map. By construction we have $\dimu N = \alpha$. Hence $\alpha \in \Delta_+(Q(a))$ by Kac's Theorem.
\end{proof}

\begin{Lemma}\label{Lemma:SourceRegularQuiver}
Let $n \in \NN$, then there exists a connected subquiver $\mathcal{Q}(n) \subseteq C_r$ such that $\mathcal{Q}(n)$ has the following properties:
\begin{enumerate}
\item $\mathcal{Q}(n)$ is source-regular with $n$ sources.
\item There is at most one sink $y  \in \mathcal{Q}(n)_0$ such that $|n_{C_r}(y) \cap \mathcal{Q}(n)_0| \not \in \{1,r\}$.
\end{enumerate}
\end{Lemma}
\begin{proof}
We prove the existence by induction on $n \in \NN$. For $n = 1$ we fix a source, say $x_1$ and let $\mathcal{Q}(n)$ be the full subquiver with vertex set $\{x_1\} \cup n_{C_r}(x_1)$.\\
For $n > 1$ we distinguish two cases. We let $\mathcal{Q}(n-1)$ be the quiver that we have constructed. If every sink $l$ in $\mathcal{Q}(n-1)_0$ satisfies $|n_{C_r}(l) \cap \mathcal{Q}(n-1)_0| \in \{1,r\}$, then we fix a sink $y$ in $\mathcal{Q}(n-1)_0$ with $|n_{C_r}(y) \cap \mathcal{Q}(n-1)_0| = 1$  and $x \in n_{C_r}(y) \setminus \mathcal{Q}(n-1)_0$. Now we define $\mathcal{Q}(n)$ to be the full subquiver with vertex set $\mathcal{Q}(n-1) \cup \{x\} \cup n_{C_r}(x)$. \\
If there exists a $($unique$)$ sink $y$ in $\mathcal{Q}(n-1)_0$ such that $1 < |n_{C_r}(y) \cap \mathcal{Q}(n-1)_0| < r$, then we fix $x \in n_{C_r}(y) \setminus \mathcal{Q}(n-1)_0$. We define $\mathcal{Q}(n)$ to be the full subquiver with vertex set $\mathcal{Q}(n-1) \cup \{x\} \cup n_{C_r}(x)$. By construction $\mathcal{Q}(n)$ has the desired properties.
\end{proof}

\begin{corollary}\label{Corollary:Maximal}
Let $n \in \NN$ and $q_n \in \NN_0$ be the number of sinks $l$ in $\mathcal{Q}(n)$ with the property $|n_{C_r}(l) \cap \mathcal{Q}(n)_0| = r$. Then $q_n = \lfloor \frac{n-1}{r-1} \rfloor$.
\end{corollary}
\begin{proof}
By construction we have $q_{n-1} \leq q_n \leq q_{n-1} + 1$ and 
\[ (\ast) \qquad q_n = \begin{cases}
q_{n-1} + 1, \ \text{if} \ (r-1) \mid n -1\\
q_{n-1}, \ \text{if} \ (r-1) \nmid n -1. 
\end{cases}
\]
We prove the statment by induction on $n$. For $n = 1$ we have $q_1 = 0$ and $n-1 = 0$. Now assume $n > 1$ and that
$q_{n-1} = \lfloor \frac{n-2}{r-1} \rfloor$. Note that $(q_{n-1}+ 2)(r-1) > n - 1$ since $r \geq 2$, hence $\lfloor \frac{n-1}{r-1} \rfloor \in \{q_{n-1},q_{n-1}+1\}$. Moreover, we have 
\[ (\ast \ast) \qquad q_{n-1}(r-1) > n - 2 - (r-1) = n -r -1.\]
We conclude
\begin{align*}
 &q_{n-1}+1 = \lfloor \frac{n-1}{r-1} \rfloor \Leftrightarrow  (q_{n-1} +1)(r-1) \leq n - 1 \\
 \Leftrightarrow & q_{n-1}(r-1) \leq n - r \stackrel{(\ast \ast)}{\Leftrightarrow} q_{n-1}(r-1) = n - r 
 \Leftrightarrow (q_{n-1} +1 )(r-1) = n - 1.
\end{align*}
Since $q_{n-1}(r-1) \leq n - 2$ we conclude 
\[  q_{n-1}+1 = \lfloor \frac{n-1}{r-1} \rfloor \Leftrightarrow (r-1)\mid (n-1).\]
In view of $(\ast)$ we get $q_n = \lfloor \frac{n-1}{r-1} \rfloor$.

\end{proof}

\section{Restrictions on Jordan Types}

\subsection{The generic rank of a representation} 

\label{Section:GenericRank}

\noindent Let $M \in \rep(\Gamma_r)$ be a representation. We consider the non-empty  open subset 
\[\MaxRk(M) := \{ [\alpha] \in \PP^{r-1} \mid \rk(M^\alpha) \ \text{is maximal} \} \subseteq \PP^{r-1}  \]
 $($see \cite[4.17]{Wor3}$)$ of the irreducible space $\PP^{r-1}$. We let $d_{M,\alpha} := \rk(M^\alpha)$ and $c_{M,\alpha}:= \dim_k M_1 - 2 d_{M,\alpha}$ for all $\alpha \in k^r \setminus \{0\}$ and set $d_{M} := \rk(M^\alpha)$ for some $[\alpha] \in \MaxRk(M)$ as well as $c_{M} := \dim_k M - 2d_{M}$. The number $d_M$ is called the \textsf{generic rank} or \textsf{maximal rank} of $M$.

 \begin{Lemma}\label{Lemma:GenericRank}
Let $0 \to A \to X \to B \to 0$ be a short exact sequence in $\rep(\Gamma_r)$. The following statements hold:
\begin{enumerate}
\item $d_{X} \geq d_A + d_B$. 
\item $c_X \leq c_A + c_B$.
\end{enumerate}
\end{Lemma}
\begin{proof}\begin{enumerate}
\item Since $\MaxRk(A)$ and $\MaxRk(B)$ are non-empty open subsets of the irreducible space $\PP^{r-1}$, we find $[\alpha] \in \MaxRk(A) \cap \MaxRk(B) \neq \emptyset$. We apply the Snake lemma and get an exact sequence
\[ 0 \to \ker A^\alpha \to \ker X^\alpha \to \ker B^\alpha \to \coker A^\alpha \to \coker X^\alpha \to \coker B^\alpha \to 0.\]
We conclude 
\begin{align*}
 d_X &\geq \rk(X^\alpha) = \rk(A^\alpha) + \rk(B^\alpha) - \dim_k \ker X^\alpha + \dim_k \ker A^\alpha + \dim_k \ker B^\alpha \\
 &\geq \rk(A^\alpha) + \rk(B^\alpha) = d_A + d_B.
\end{align*}

\item We have $c_X = \dim_k X - 2d_X \leq \dim_k A + \dim_k B - 2(d_A + d_B) = c_A + c_B$.
\end{enumerate}
\end{proof}

\begin{Example}\label{Example:EKPEIP} \normalfont
Let $M \in \rep(\Gamma_r)$ be an indecomposable representation in $\EKP \cup \EIP$, then $d_M = \min\{ \dim_k M_1,\dim_k M_2\}$ and $(d_M,d_M+c_M) \in \{(\dim_k M_1,\dim_k M_2), (\dim_k M_2,\dim_k M_1) \}$. We conclude with Kac's Theorem $q_{\Gamma_r}(d_M,d_M+c_M) \leq 1$. 
\end{Example}

\subsection{Elementary representations}

In this section we study the generic rank of the so-called elementary representations of $\Gamma_r$. By definition the set of all elementary representations is the smallest subset $\cE \subseteq \rep(\Gamma_r)$ such that every regular representation $X$ has a filtration 
with all filtration factors in $\cE$. We show that $q_{\Gamma_r}(d_E,d_E+c_E) < 1$ for each elementary representation $E$ and conclude that each regular representation $M$ satisfies $q_{\Gamma_r}(d_M,d_M+c_M) < 1$. \noindent For basic properties and results on hereditary algebras and elementary representations, used in the following proofs, we refer to \cite{Ker3}, \cite[1.3]{KerLuk1} and \cite[A1]{Ri10}.

\begin{Definition} \cite[1]{KerLuk1} 
Let $r \geq 3$. A non-zero regular representation $E \in \rep(\Gamma_r)$ is called \textsf{elementary}, if there is no short exact sequence $0 \to A \to E \to B \to 0$ with $A$ and $B$ regular non-zero. We denote by $\cE \subseteq \rep(\Gamma_r)$ the set of all elementary representations.
\end{Definition}

\begin{Remark}
Observe that $d_M = 0$ if and only if $M$ is semisimple. In particular, $d_M \neq 0$ for each regular representation $M$.
\end{Remark}

\begin{Lemma}\label{Corollary:QuadraticFormRegular}
Let $r \geq 3$ and $K \in \RR$.
\begin{enumerate}
\item If $\frac{d_E+c_E}{d_E} < K$ for each $E \in \cE$, then $\frac{d_X+c_X}{d_X} < K$  for each $($not necessarily indecomposable$)$ regular representation $X \in \rep(\Gamma_r)$.
\item If $q_{\Gamma_r}(d_E,d_E+c_E) < 1$ for each $E \in \cE$, then $q_{\Gamma_r}(d_X,d_X+c_X) < 1$ for each $($not necessarily indecomposable$)$ regular representation $X \in \rep(\Gamma_r)$.
\end{enumerate} 
\end{Lemma}
\begin{proof}
$(a)$ Let $X$ be a regular representation, then there is a filtration of minimal length $n \in \NN$
\[0 = X_0 \subset X_1 \subset \ldots \subset X_n = X\]
such that $X_i/X_{i-1} \in \cE$ for all $1 \leq i \leq n$. We prove the statement by induction on $n \in \NN$. Cleary, $n = 1$ if and only if $X$ is elementary. Now let $n > 1$, then 
we have a short exact sequence $0 \to X_{n-1} \to X \to X/X_{n-1} \to 0$ such that $A := X_{n-1}$ has a filtration of length $\leq n-1$ with filtration factors in $\cE$ and $B:= X/X_{n-1}$ is elementary.
Hence $\frac{d_A+c_A}{d_A} < K, \frac{d_B+c_B}{d_B} < K$. We conclude with Lemma $\ref{Lemma:GenericRank}$
\[
d_X+c_X \leq d_X + c_A + c_B < d_X + (K-1)(d_A+d_B) \leq d_X + (K-1)d_X \leq K d_X
.\]
$(b)$ This follows from $(a)$ and Example $\ref{Examples}$ wit $K = \frac{r+\sqrt{r^2-4}}{2}$.
\end{proof}

\begin{proposition}\cite[Appendix 1]{Ri10}\label{Proposition:RingelElementary}
Let $E$ be non-zero regular module $E$. Then $E$ is elementary if and
only if given any submodule $U$ of $E$, the submodule $U$ is preprojective or the factor module
$E/U$ is preinjective.
\end{proposition}

\begin{Lemma}\label{Lemma:DimensionVectorNotInEKP}
Let $r \geq 3$, $M \in \rep(\Gamma_r)$ be an elementary representation such that $\dimu M = (a,b)$ with $a \leq b$ and $M \not \in \EKP$, then $a \leq b \leq r -1$.
\end{Lemma}
\begin{proof}
We assume that $M$ does not have the equal kernels property.  By $\ref{Theorem:Worch}$ we find $\alpha \in k^r \setminus \{0\}$ and $f \in \Hom(X_\alpha,M) \setminus \{0\}$. Since $\dimu X_\alpha = (1,r-1)$ and $M$ is regular indecomposable, there is $d \in \{1,\ldots,r-1\}$ such that $\dimu \im f = (1,d)$ $($see \cite[3.5, 3.8]{Bi3}$)$. It follows that $\im f$ is regular $($see \cite[VIII.2.13]{Assem1} and $\ref{Lemma:PropertiesARquiver})$. Since $M$ is elementary, we conclude with $\ref{Proposition:RingelElementary}$ that $M / \im f$ is preinjective. Hence Lemma $\ref{Lemma:PropertiesARquiver}$ implies that $x \geq y$, where $(x,y) := \dimu M/\im f = (a,b) - (1,d)$. Hence $b - a < r - 1$.
We assume that $b \geq r$, then $b-a < r - 1 < b$ and \cite[14.7]{Bi3} $($see also \cite[3.2]{Ri10}$)$ implies the existence of a non-preprojective representation $U \subseteq M$ with dimension vector $(1,r-1)$. We conclude with $\ref{Proposition:RingelElementary}$ that $M/U$ is a preinjective representation with dimension vector $(a-1,b-(r-1))$ and $b-(r-1) \neq 0$. Now \cite[14.9]{Bi3} implies
\[\frac{b-1}{b-(r-1)} \geq \frac{a-1}{b-(r-1)} \geq \frac{r+\sqrt{r^2-4}}{2} > (r-1).\] 
It follows $r(r-2) > b(r-2)$ and therefore $r > b$, a contradiction. Hence $a \leq b \leq r-1.$
\end{proof}

\begin{Lemma}\label{Lemma:QuadraticFormRestriction}
Let $r \geq 3$, $M \in \rep(\Gamma_r)$ and $N := M_{\mid \{\gamma_1,\ldots,\gamma_{r-1}\} } \in \rep(\Gamma_{r-1})$ be the restriction of $M$ to $\Gamma_{r-1}$. The following statements hold.
\begin{enumerate}
\item $d_{M} \geq d_{N}$ and $c_M \leq c_N$.
\item If $q_{\Gamma_{r-1}}(d_N,d_N+c_N) < 1$, then $q_{\Gamma_{r}}(d_M,d_M+c_M) < 1$.
\end{enumerate}
\end{Lemma}
\begin{proof}
\begin{enumerate}

\item Let $\alpha \in k^{r-1} \setminus \{0\}$ such that $d_N = \rk(N^\alpha)$, then $\beta := (\alpha_1,\ldots,\alpha_{r-1},0) \in k^r \setminus \{0\}$ and $d_M \geq \rk(M^\beta) = \rk(N^\alpha) = d_N$.
\item Note that the statement is obviously true if $N$ and therefore $M$ are zero. Hence we can assume that $N$ is not zero, i.e. $(c_N,d_N) \neq (0,0)$. Hence $q_{\Gamma_{r-1}}(d_N,d_N+c_N) < 1$ implies $d_N \neq 0$. Since $\frac{{r-1}+\sqrt{(r-1)^2-4}}{2} < \frac{r+\sqrt{r^2-4}}{2}$, we conclude with Example $\ref{Examples}$ and $(a)$ that
\[\frac{r+\sqrt{r^2-4}}{2} > \frac{d_N+c_N}{d_N}  = 1 + \frac{c_N}{d_N} \geq 1 + \frac{c_M}{d_M} =  \frac{d_M +c_M}{d_M},\]
i.e. $q_{\Gamma_r}(d_M,d_M+c_M) < 1$.
\end{enumerate}
\end{proof}

We denote by $D_{\Gamma_r} \colon \rep(\Gamma_r) \to \rep(\Gamma_r)$ the duality given by $D_{\Gamma_r}M_i := \Hom_k(M_{3-i},k)$ for $i \in \{1,2\}$ and $D_{\Gamma_r}M(\gamma_i) := M(\gamma_i)^\ast$ for all $1 \leq i \leq r$. The duality sends elementary representations to elementary representations. Moreover, $M \in \EKP$ if and only if $D_{\Gamma_r} M \in \EIP$. The proof of the following result is inspired by the arguments used in \cite[4.1]{Ri10}.

\begin{Lemma}\label{Lemma:RestrictionRegular}
Let $r \geq 3$, $E \in \rep(\Gamma_r)$ elementary and $\dimu E = (a,b)$ with $2 \leq a \leq b \leq r - 1$. We consider the restriction $M := E_{\mid \{\gamma_1,\ldots,\gamma_{r-1}\}}$ of $E$ to $\Gamma_{r-1}$.  Then $M \in \rep(\Gamma_{r-1})$ is regular, i.e. every indecomposable direct summand of $M$ is regular in $\rep(\Gamma_{r-1})$.
\end{Lemma}
\begin{proof}
Let $U \subseteq M$ be a non-zero direct summand of $M$ and assume that $U$ is preinjective or preprojective. Since $\dimu M = (a,b)$ and $a,b \leq r-1$, we conclude with the considerations in section $\ref{Section:ARquiver}$ that $U \in \{I_1,P_1,I_2,P_2\}$. Recall that $P_1,P_2$ are projective and $I_1,I_2$ are injective with dimension vector $(0,1), (1,r-1), (1,0)$ and $(r-1,1)$, respectively  $($note that we consider $\rep(\Gamma_{r-1}))$.\\
Assume that $I_1$ is a direct summand of $M$. Then there is a $1$-dimensional $k$-subspace $U \subseteq E_1$ such that $E(\gamma_i)_{|U} = 0$ for all $1 \leq i \leq r -1$. If $E(\gamma_r)_{\mid U} = 0$, then $I_1$, considered as a representation of $\Gamma_r$, is simple and injective and therefore a direct summand of $E$, a contradiction. Hence $E(\gamma_r)_{|U} \neq 0$. Now $E(\gamma_i)_{|U} = 0$ for all $1 \leq i \leq r -1$ implies that $U$ generates a subrepresentation $M(U) \subseteq E \in \rep(\Gamma_r)$ with dimension vector $(1,1)$. Since $M(U)$ is not preprojective, we conclude with Proposition $\ref{Proposition:RingelElementary}$ that $E/M(U)$ with dimension vector $(a-1,b-1)$ is a preinjective representation. Since $0 < a-1 \leq b-1$, this is a contradiction to Lemma $\ref{Lemma:PropertiesARquiver}$.\\
Assume now that $P_1$ is a direct summand of $M$, by duality we conclude that $I_1$ is isomorphic to a direct summand of $N := (D_{\Gamma_r}E)_{\mid \{\gamma_1,\ldots,\gamma_{r-1} \}}$. Then there is a $1$-dimensional $k$-subspace $U \subseteq (D_{\Gamma_r}E)_1 = \Hom_k(E_2,k)$ such that $D_{\Gamma_r}E(\gamma_i)_{|U} = 0$ for all $1 \leq i \leq r -1$. If $D_{\Gamma_r}E(\gamma_r)_{|U} = 0$, then $I_1$ is a direct summand of $D_{\Gamma_r}E$, a contradiction. Hence $D_{\Gamma_r}E(\gamma_r)_{|U} \neq 0$. Now $D_{\Gamma_r}E(\gamma_i)_{|U} = 0$ for all $1 \leq i \leq r -1$ implies that $U$ generates a subrepresentation $M(U) \subseteq D_{\Gamma_r}E$ with dimension vector $(1,1)$. We conclude with $\ref{Proposition:RingelElementary}$ that $(b,a) - (1,1) = (b-1,a-1)$ is the dimension vector of a preinjective representation in $\rep(\Gamma_r)$. Since $1 \leq b-1 \leq r - 2 < r$, we get with Lemma $\ref{Lemma:PropertiesARquiver}(d)$ that $b-1 = 1$ and therefore $a - 1  = 0$, a contradiction since $a \geq 2$.\\
Assume now that $I_2$ is a direct summand of $M$, then $r-1 \geq a$ and therefore $a = r - 1$ and $b = r -1$. We write $M = I_2 \oplus U$, then $\dimu U = (0,r-2)$ and $U$ has $P_1$ as a direct summand, a contradiction.\\
Assume that  $P_2$ is a direct summand of $M$, then $r-1 \geq b$ and therefore $b = r - 1$ and $2 \leq a \leq r -1$. We write $M = P_2 \oplus U$, then $\dimu U = (a-1,0)$ and $I_1$ is as a direct summand of $U$, a contradiction.
\end{proof}

\begin{proposition}\label{Lemma:ElementaryQuadratric}
Let $r \geq 3$ and $E \in \rep(\Gamma_r)$ be an elementary representation. Then $q_{\Gamma_r}(d_E,d_E+c_E) < 1$.
\end{proposition}
\begin{proof} 
In view of Example $\ref{Example:EKPEIP}$ and duality, we can assume that $\dim_k E_1 \leq \dim_k E_2$ and $E \not\in \EKP$. Now $\ref{Lemma:DimensionVectorNotInEKP}$ implies that $1 \leq \dim_k E_1 \leq \dim_k E_2 \leq r-1$.
If $\dim_k E_1 = 1$, then we have $d_E = 1$ and $c_E = 1+\dim_k E_2 -2d_E = \dim_k E_2 - 1$. We conclude with Example $\ref{Examples}$ that 
$q_{\Gamma_r}(d_E,d_E+c_E) = q_{\Gamma_r}(\dim_k E_1,\dim_k E_2) < 1$. \\
Now we assume that $\dim_k E_1 \geq 2$ and let $(a,b) := \dimu E$. We do the proof by induction on $r \geq 3$. \\
For $r = 3$ we have $a = 2 = b$. Now $\ref{Lemma:RestrictionRegular}$ implies that $M := E_{\mid \{\gamma_1,\gamma_{2}\}}$ is regular in $\rep(\Gamma_2)$. The regular representations of $\Gamma_2$ are known $($see for example \cite[XI.4.3]{Assem2}$)$ and one has $d_M = 2$ and $c_M = 0$. We conclude with Lemma $\ref{Lemma:QuadraticFormRestriction}(a)$ $2 \geq d_E \geq d_M = 2$ and therefore $d_E = 2$. It follows $q_{\Gamma_3}(d_E,d_E+c_E) = q_{\Gamma_3}(2,2) < 1$.\\
Now we assume that $r > 3$. In view of $\ref{Lemma:RestrictionRegular}$ we know that $M := E_{\mid \{\gamma_1,\ldots,\gamma_{r-1}\}}$ decomposes into regular direct summands in $\rep(\Gamma_{r-1})$. The inductive hypothesis and $\ref{Corollary:QuadraticFormRegular}(b)$ imply $q_{\Gamma_{r-1}}(d_M,d_M+c_M) < 1$ and $\ref{Lemma:QuadraticFormRestriction}(b)$ yields $q_{\Gamma_{r}}(d_E,d_E+c_E) < 1$.
\end{proof}

\subsection{Restrictions on Jordan types of indecomposable representations}
Now we prove the main result of this section.

\begin{Theorem}\label{Theorem:ConstantJordanTypeRestrictions}
Let $r \geq 2$ and $M \in \rep(\Gamma_r)$ be indecomposable. The following statements hold.
\begin{enumerate}
\item We have $q_{\Gamma_r}(d_M,d_M + c_M) \leq 1$.
\item If $M$ is not simple and of constant Jordan type $[1]^c [2]^d$, then $(c,d) \in \IJT$.
\item If $N_1,\ldots,N_l \in \rep(\Gamma_r)$ are regular indecomposable and $N := N_1 \oplus \cdots \oplus N_l$ is of constant Jordan type $[1]^c [2]^d$, then $q_{\Gamma_r}(d,d+c) < 1$ and $c \geq l(r - 1)$.

\end{enumerate}
\end{Theorem}
\begin{proof}
\begin{enumerate}
\item If $M$ is preprojective or preinjective, then $M \in \EKP \cup \EIP$ by \cite[2.7]{Wor1} and we conclude with $\ref{Example:EKPEIP}$ that $q_{\Gamma_r}(d_M,d_M+c_M) \leq 1$. If $M$ is regular, Proposition $\ref{Lemma:ElementaryQuadratric}$ and $\ref{Corollary:QuadraticFormRegular}$ imply $q_{\Gamma_r}(d_M,d_M+c_M) < 1$.
\item We have $d = d_M$ and $c_M = c$ and conclude with $(a)$ that $q_{\Gamma_r}(d,d+c) \leq 1$. 
 Since $M$ is not simple we have $d > 0$ and conclude with \cite[10.1.4]{Be1} that $c \geq r -1$.
 \item Again Proposition $\ref{Lemma:ElementaryQuadratric}$ and $\ref{Corollary:QuadraticFormRegular}$ imply $q_{\Gamma_r}(d_N,d_N+c_N) < 1$. Since $N$ has constant Jordan type, every representation $N_i$ has constant Jordan type $($see \cite[5.1.9]{Be1}$)$, say $[1]^{c_i} [2]^{d_i}$. We conclude with \cite[10.1.4]{Be1} that $c_i \geq (r-1)$ and therefore $c \geq l(r-1)$.

\end{enumerate}

\end{proof}

\begin{Remark}
Consider the projective indecomposable representation $P_2$  with dimension vector $(1,r)$. Then $P_2 \oplus P_2$ has constant Jordan type $[1]^{2r-2}[2]^2$ and $q_{\Gamma_r}(2,2+2r-2) = q_{\Gamma_r}(2,2r) = 4 + 4r^2 - 4r^2 = 4$. This shows that Theorem $\ref{Theorem:ConstantJordanTypeRestrictions}(c)$ does not hold for arbitrary representations that are not semisimple.
\end{Remark}

\section{Existence of constant Jordan types}

\noindent In this section we determine all positive roots $(a,b) \in \NN_0 \times \NN_0$ of $q_{\Gamma_r}$ that are the dimension vector of an indecomposable representation $M$ in $\EKP$ and draw conclusions for the constant Jordan types, that can realized in $\rep(\Gamma_r)$ by indecomposable representations.

\begin{Definition}
A pair $(c,d) \in \NN_0 \times \NN_0$ is called
\begin{enumerate}
\item $\EKP$-admissable, provided there exists $M \in \rep(\Gamma_r)$ indecomposable in $\EKP$ of constant Jordan type $[1]^c [2]^d$. We define 
\[ \Ad(\EKP) := \{ (c,d) \in \NN_0 \times \NN_0  \mid (c,d) \ \text{is} \ \EKP\text{-admissable}\}. \]
\item $\EIP$-admissable, provided there exists $M \in \rep(\Gamma_r)$ indecomposable in $\EIP$ of constant Jordan type $[1]^c [2]^d$. We define 
\[ \Ad(\EIP) := \{ (c,d) \in \NN_0 \times \NN_0 \mid (c,d) \ \text{is} \ \EIP\text{-admissable} \}. \]
\end{enumerate}
\end{Definition}

It is not hard to see that $M \in \EKP$ if and only if $D_{\Gamma_r}M \in \EIP$ for each $M \in \rep(\Gamma_r)$ $($see \cite[2.1.1]{Wor1}$)$. Moreover, $M$ has constant Jordan type $[1]^c [2]^d$ if and only if $D_{\Gamma_r}M$ has constant Jordan type $[1]^c [2]^d$. We conclude that $\Ad(\EKP)  = \Ad(\EIP)$ and define $\Ad := \Ad(\EKP)$.

\begin{Definition}
We define 
 \[  \dimu \EKP := \{ \dimu M \mid M \in \EKP \text{indecomposable} \}.\]

\end{Definition}

\begin{Lemma}\label{Lemma:Bijection}
The assignment
\[ \Xi \colon \Ad \to \dimu \EKP, (c,d) \mapsto (d,d+c)\]
is a well-defined bijection.
\end{Lemma}
\begin{proof}
Let $(c,d) \in \Ad$ and $M \in \EKP$ be indecomposable with constant Jordan type $(c,d)$. Since $M \in \EKP$, we have $\dim_k M_1 = \rk M = d$ and conclude $\dim_k M_2 = c + 2d - d = c + d$. Hence $(d,d+c) \in \dimu \EKP$. \\
Now let $M \in \EKP$ be indecomposable with dimension vector $(a,b)$. Then $M$ has constant Jordan type $[1]^{b-a} [2]^a$ and therefore $(b-a,a) \in \Ad$ with $\Xi(b-a,a) = (a,b-a+a) = (a,b)$.
\end{proof}

\noindent The main result of this section is the following:

\begin{Theorem}\label{Theorem:ExistenceJordanTypes}
We have
\[ \dimu \EKP =  \{ (a,b) \in \NN^2 \mid q_{\Gamma_r}(a,b) \leq 1, b-a \geq r-1\} \cup \{(0,1)\} = :\cD.\]
\end{Theorem}

\noindent In view of Lemma $\ref{Lemma:Bijection}$ and Theorem $\ref{Theorem:ConstantJordanTypeRestrictions}$ we have 
\begin{align*}
 \dimu \EKP &= \Xi(\Ad) \subseteq \Xi(\{ (c,d) \in \NN^2 \mid q_{\Gamma_r}(d,d+c) \leq 1, c \geq r-1\} \cup \{(1,0)\}) \\
 & =\{ (d,d+c) \mid c,d \in \NN, q_{\Gamma_r}(d,d+c) \leq 1, c \geq r-1\} \cup \{(0,1)\} \\
& =\{ (a,b) \in \NN^2 \mid q_{\Gamma_r}(a,b) \leq 1, b-a \geq r-1\} \cup \{(0,1)\} = \cD.
\end{align*}
From now on we fix $(a,b) \in \cD \subseteq \NN_0 \times \NN$ and show that there exists an indecomposable representation in $\EKP$ with dimension vector $(a,b)$. Recall that the roots $(a^\prime,b^\prime)$ with $a^\prime \leq b^\prime$ and $q_{\Gamma_r}(a,b) = 1$ correspond to the indecomposable preprojective representations of $\Gamma_r$. These are known to have the equal kernels property $($\cite[2.2.3]{Wor1}$)$. Therefore we only have to consider the case that $q_{\Gamma_r}(a,b) \leq 0$. Then it follows from the definition of $\cD$ that $a \geq 2$.
 We write $b = qa  +s$ with $q \in \NN_0$ and $0 \leq s < a$.

\begin{Lemma}
We have $0 \leq q \leq r-1$. In particular, $2 \leq a < b < ra$.
\end{Lemma}
\begin{proof}
We have
\begin{align*}
0 \geq q_{\Gamma_r}(a,qa+s) &= a^2 + q^2 a^2 + 2 q a s + s^2 - ra(qa+s)\\
&= a^2(1+q^2-rq) + a(2qs-rs)+s^2.
\end{align*}
If $q \geq r$, we get $1+q^2-rq \geq 1$ and $(2qs-rs) \geq 0$, hence $0 \geq q_{\Gamma_r}(a,b) \geq a^2 +  s^2 \stackrel{a \geq 2}{\geq} 4$,  a contradiction.
\end{proof}

\noindent  Since $b -a \geq  r -1$, Example $\ref{Examples}$ shows that we only have to consider the case $r \geq 3$. Hence we assume from now on that $(a,b) \in \NN^2$ and $($see Example $\ref{Examples})$
\[ \boxed{2 \leq a < b < (\frac{r+\sqrt{r^2-4}}{2})a, b-a \geq r-1 \ \text{and} \ r \geq 3.}\]

\subsection{Chen's approach}

\noindent We modify the arguments used in \cite{BoChen1} to prove that there exists an indecomposable representation $F_{(a,b)} \in \EKP$ with dimension vector $(a,b)$ if $b \leq (r-1)a$.

\begin{Lemma}\label{Lemma:KernelEcholonEKP}
Let $a^\prime,b^\prime \in \NN$ such that $b^\prime-a^\prime \geq r - 1$, so in particular $b^\prime > a^\prime > 0 $. For $1 \leq l \leq b^\prime-a^\prime+1$ we denote with $I(l)$ the $b^\prime \times a^\prime$-matrix 
\[  
I(l) = \begin{pmatrix}
0_{l-1\times a^\prime} \\
I_{a^\prime} \\
0\\
\vdots \\
0
\end{pmatrix}.
\]
Fix $J \subseteq \{1,\ldots,b^\prime-a^\prime+1\}$ such that $|J| = r$. Let $\varphi \colon \{1,\ldots,r\} \to J$ be a bijection and define a representation $M_\varphi = (k^{a^\prime},k^{b^\prime},M(\gamma_i))$ via $M_\varphi(\gamma_i)(x) := I(\varphi(i))x$ for all $1\leq i \leq r$ and all $x \in k^{a^\prime}$. 
Then $M_\varphi$ has the equal kernels property.
\end{Lemma}
\begin{proof}
Let $\alpha \in k^r \setminus  \{0\}$. We have to check that $M_\varphi^\alpha = \sum^r_{i=1} \alpha_i M_\varphi(\gamma_i) \colon k^{a^\prime} \to k^{b^\prime}$ is injective, i.e. $\rk(M^\alpha) = a^\prime$. Let $\cE_{a^\prime}$ and $\cE_{b^\prime}$ the canonical basis of $k^{a^\prime}$ and $k^{b^\prime}$, respectively. Then $M^{\cE_{a^\prime}}_{\cE_{b^\prime}}(M_\varphi^\alpha)$ is in echelon form and of rank $a^\prime$.
\end{proof}

\noindent Recall that $M \in \rep(\Gamma_r)$ is a brick, provided $\End_k(M) = k$. Chen constructed in \cite[3.6]{BoChen1} for each root $(x,y)$ a brick $M_{(x,y)}$ such that $\dimu M_{(x,y)}  = (x,y) $. We combine his construction with Lemma $\ref{Lemma:KernelEcholonEKP}$ to show:

\begin{proposition}\label{Result1}
Assume that $b \leq (r-1)a$. Then there exists a brick $F_{(a,b)} \in \rep(\Gamma_r)$ such that $F_{(a,b)} \in \EKP$ and $\dimu F_{(a,b)} = (a,b)$.
\end{proposition}
\begin{proof}
We consider $b = qa+s$ with $q \in \{0,\ldots,r-2\}$ and $s < a$ or $q = r-1$ and $s = 0$.
We distinguish the following cases:
\begin{enumerate}

\item $q \leq 1$ and $s < r - 1$, then $b = qa + s < a + r - 1$, a contradiction since $b - a \geq r -1$. 
\item $q = 1$ and $r-1 \leq s < a$. Note that $s \geq r-1$ implies $s+1 \not \in \{1,2\}$. We extend the map $1 \mapsto 1, 2 \mapsto s + 1 = b -a +1, 3 \mapsto 2$ to an injection $\varphi \colon \{1,\ldots,r\} \to \{1,\ldots,b-a+1\}$. In view of $\ref{Lemma:KernelEcholonEKP}$ $M_\varphi$ has the equal kernels property and \cite[3.6(2)]{BoChen1} implies that $M_\varphi$ is a brick and therefore indecomposable.
\item $2 \leq q \leq r-1$, $s=0$. Note $2 \neq (i-1)a+1 \leq b - a + 1$ for all $1 \leq i \leq q$ since $a \geq 2$.\\
We have $q+1 \leq r$ and extend the map 
\[ \{1,\ldots,q+1\} \to \{1,\ldots,b-a+1\}, i \mapsto \begin{cases}
(i-1)a+1, \ \text{for} \ i \in \{1,\ldots,q\} \\
2, \ \text{for} \ i = q+1
\end{cases}
\]
to an injection $\varphi \colon \{1,\ldots,r\} \to \{1,\ldots,b-a+1\}$. Again apply $\ref{Lemma:KernelEcholonEKP}$ and \cite[3.6(3)]{BoChen1}.
\item $2 \leq q \leq r-2$ $(r \geq 4)$ and $0 < s < a$. We extend the map 
\[ \{1,\ldots,q+2\} \to \{1,\ldots,b-a+1\}, i \mapsto \begin{cases}
(i-1)a+1, \ \text{for} \ i \in \{1,\ldots,q\} \\
qa+s+1, \ \text{for} \ i = q + 1 \\
2, \ \text{for} \ i = q +2
\end{cases}
\]
 to an injection $\varphi \colon \{1,\ldots,r\} \to \{1,\ldots,b-a+1\}$ and conclude that $M_\varphi$ has the equal kernels property.  Apply $\ref{Lemma:KernelEcholonEKP}$ and \cite[3.6(4)]{BoChen1}.
\end{enumerate}
\end{proof}

\subsection{The case $(r-1)a + 1 \leq b < (\frac{r+\sqrt{r^2-4}}{2})a$}

Although Chen shows the existence of a brick with dimension vector $(a,b)$ for each root $(a,b)$, we can not use his arguments for the case $b > (r-1)a$, as the following example shows: 

\begin{example}
We consider the case $r = 3$ and $(a,b) = (2,5)$. Then $b > (r-1)a$. The only element $(a^\prime,b^\prime)$ in the Coxeter orbit of $(2,5)$ with $a^\prime \leq b^\prime \leq (r-1)a^\prime$ or $b^\prime \leq a^\prime \leq (r-1)b^\prime$ is $(1,1)$. But we will not find an indecomposable representation in $\EKP$ with this dimension vector.
\end{example}

\vspace{0,5cm}
To prove the existence of indecomposable representations for $(r-1)a + 1 \leq b < a (\frac{r+\sqrt{r^2-4}}{2})$, we consider the universal cover $C_r$ of the quiver $\Gamma_r$. We let $(C_r)^+$ be the set of all sources of $C_r$, $(C_r)^-$ be the set of all sinks and denote with $\rep(C_r)$ the category of finite dimensional representations of $C_r$. For the sake of simplicity we only recall the most important properties. For a more detailed description we refer to \cite{Gab3},\cite{Ri7} and \cite{Bi2}.\\ We fix a covering $\pi \colon C_r \to\Gamma_r$ of quivers, i.e. $\pi$ is a morphism of quivers and for each $x \in (C_r)_0$ the induced map $n_{C_r}(x) \to n_{\Gamma_r}(\pi(x))$ is bijective.

\noindent By \cite[3.2]{Gab2} there exists an exact functor $\pi_\lambda \colon \rep(C_r) \to \rep(\Gamma_r)$ such that $\pi_{\lambda}(M)_{1} = \bigoplus_{x \in (C_r)^+} M_x$, $\pi_{\lambda}(M)_{2} = \bigoplus_{y \in (C_r)^-} M_y$ and $\pi_\lambda(M)(\gamma_i) = \bigoplus_{\delta \in \pi^{-1}(\gamma_i)} M(\delta)$ for all $i \in \{1,\ldots,r\}$. Morphisms are defined in the obvious way. 

\begin{Theorem}\cite[3.6]{Gab3}, \cite[6.2,6.3]{Ri7}\label{TheoremRingelGabriel}
There exists a free group $G$ of rank $r-1$, that acts on $\rep(\Gamma_r)$ such that the following statements hold:
\begin{enumerate}[topsep=0em, itemsep= -0em, parsep = 0 em, label=$(\alph*)$]
\item $\pi_{\lambda}$ sends indecomposable representations in $\rep(C_r)$ to indecomposable representations in $\rep(\Gamma_r)$.
\item If $M \in \rep(C_r)$ is indecomposable, then $\pi_{\lambda}(M) \cong \pi_{\lambda}(N)$ if and only if $M^g \cong N$ for some $g \in G$.
\item The category $\rep(C_r)$ has almost split sequences, $\pi_{\lambda}$ sends almost split sequences to almost split sequences and $\pi_{\lambda}$ commutes with the Auslander-Reiten translates, i.e. 
\[ \tau_{\Gamma_r} \circ \pi_{\lambda} = \pi_{\lambda} \circ \tau_{C_r}.\]
\end{enumerate}
\end{Theorem}

\noindent The next result tells us that it is not hard to decide whether the push-down $\pi_\lambda(M)$ of a representation $M \in \rep(C_r)$ has the equal kernels property.

\begin{Theorem}\label{Theorem:INJEKP}\cite[4.1]{Bi2}
\label{PropositionGeneral}
Let $M \in \rep(C_r)$ be an indecomposable representation. The following statements are equivalent:
\begin{enumerate}[topsep=0em, itemsep= -0em, parsep = 0 em, label=$(\alph*)$]
\item $N := \pi_{\lambda}(M) \in \EKP$.
\item $N(\gamma_i)$ is injective for all $i \in \{1,\ldots,r\}$.
\item $M \in \Inj := \{ M \in \rep(C_r) \mid \forall \delta \in (C_r)_1: M(\delta) \ \text{is injective}\}$.
\end{enumerate}
\end{Theorem}

\noindent For the sake of book-keeping, recall that we assume $r \geq 3$, $q_{\Gamma_r}(a,b) \leq 0$, $b - a \geq r - 1$, $a \neq 0$ and that $a = 1$ implies $b = r$.  Since $q_{\Gamma_r}(1,r) = 1$, we have $a \geq 2$. Hence we get
\[ (r-1)a + 1 \leq (r-\frac{1}{r-1})a.\]
Moreover, $r \geq 3$ implies 
\[ r - \frac{1}{r-1} < \frac{r+\sqrt{r^2-4}}{2} .\]

\noindent We distinguish therefore the 
cases

\[ \boxed{(r-1)a+1 \leq b \leq (r-\frac{1}{r-1})a} \ \text{and} \  \boxed{(r-\frac{1}{r-1})a < b < (\frac{r+\sqrt{r^2-4}}{2})a.}\]

\bigskip

\subsubsection{The case $(r-1)a+1 \leq b \leq (r-\frac{1}{r-1})a$}

The aim of this section is to show the existence of an indecomposable representation $E_{(a,b)} \in \rep(C_r)$ such that $E_{(a,b)} \in \Inj$ and $\dimu \pi_{\lambda}(E_{(a,b)}) = (a,b)$.

\begin{Lemma}\label{Lemma:MaximalNumber}
Let $q := \lfloor \frac{a - 1}{r-1} \rfloor$ and $s \in \{0,\ldots,r-2\}$ such that $q(r-1)+s = a - 1$. Then $m := 
(r-1)a+ 1 + q(r-2)+s - 1  =\lfloor (r - \frac{1}{r-1})a \rfloor$. 
\end{Lemma}
\begin{proof}
We have for $z \in \{0,1\}$
\[ m + z = (r-1)a  +1 + q(r-1) -q +s  - 1  + z = ra - q - 1 + z = (r - \frac{q+1-z}{a})a,\]
and conclude 
\begin{align*}
m+z \leq (r-\frac{1}{r-1})a &\Leftrightarrow -\frac{q+1-z}{a} \leq - \frac{1}{r-1} \Leftrightarrow qr-q+r-1-zr+z\geq a \\
&  \Leftrightarrow q(r-1) -  (z-1)(r-1) \geq a \Leftrightarrow (q-z+1)(r-1) > a - 1 \\
& \Leftrightarrow z = 0.
\end{align*}
\end{proof}

\begin{proposition}\label{Result2}
Assume that $(r-1)a+1 \leq b \leq (r-\frac{1}{r-1})a$. There exists an indecomposable representation $E_{(a,b)} \in \rep(C_r)$ such that $E_{(a,b)} \in \Inj$ and $\dimu \pi_\lambda E_{(a,b)} = (a,b)$.
\end{proposition}
\begin{proof}
Let $\mathcal{Q}(a)$ be the quiver constructed in Lemma $\ref{Lemma:SourceRegularQuiver}$, $q \in \NN_0$ the number of sinks $y$ in $\cQ(a)_0$ with $|n_{C_r}(y) \cap \cQ(a)_0| = r$. In view of $\ref{Corollary:Maximal}$ we find $s \in \{0,\ldots,r-2\}$ such that $q(r-1) + s = a - 1$. Note that $s \neq 0$ if and only if there exists a $($uniquely determined$)$ sink in $y_0 \in \cQ(a)_0$ such that $1 < |n_{C_r}(y_0) \cap \cQ(a)_0| < r$ and $|n_{C_r}(y_0) \cap \cQ(a)_0| = s + 1$.
Let $y_1,\ldots,y_q$ be the sinks that satisfy $|n(y_i) \cap \cQ(a)_0| = r$. 
In view of the assumption and $\ref{Lemma:MaximalNumber}$ we have 
\[(r-1)a+1 \leq b \leq (r-1)a+1 + q(r-2)+s-1.\] Hence we have 
\[ 0 \leq  b - (r-1)a -1  \leq q(r-2)+s-1.\] 
\begin{enumerate}
\item If $s \neq 0$, we find for $i \in \{1,\ldots,q\}$ an element $\beta_i \in \{0,\ldots,r-2\}$ and $\beta_{q+1} \in \{0,\ldots,s-1\}$ such that 
$b - (r-1)a-1 = \sum^{q+1}_{i=1} \beta_i$.
We define $\alpha \in \NN^{\cQ(a)_0}$ by setting
\[ \alpha_l := \begin{cases}
 1, l \ \text{source} \\
1, l \ \text{sink and} \ l \notin \{y_1,\ldots,y_{q+1}\} \\
1 +\beta_l, l \in \{y_1,\ldots,y_{q+1}\}.
\end{cases}
\]

\item If $s = 0$, we find for $i \in \{1,\ldots,q\}$ an element $\beta_i \in \{0,\ldots,r-2\}$ such that 
\[b - (r-1)a-1 = \sum^{q}_{i=1} \beta_i.\]
We define $\alpha \in \NN^{\cQ(a)_0}$ by setting
\[ \alpha_l := \begin{cases}
1, l \ \text{source} \\
1, l \ \text{sink and} \ l \notin \{y_1,\ldots,y_q\} \\
1 + \beta_l, l \in \{y_1,\ldots,y_{q}\}.
\end{cases}
\]

\end{enumerate} 

By construction $\alpha$ satisfies $\supp(\alpha) = \cQ(a)_0$. For each source $l$ we have $\alpha_l = 1$ and for each sink $j$ we have $\alpha_j  \leq \max \{1,|n_{a}(j)|-1\}$. Hence we conclude with Lemma $\ref{Lemma:ImaginaryRoot}$ that $\alpha \in \Delta^+(\cQ(a))$ and Theorem $\ref{Theorem:Kac}$ implies that we find an indecomposable representation $E_\alpha \in \rep(\cQ(a)) \subseteq \rep(C_r)$ with dimension vector $\alpha$. The pushdown $\pi_{\lambda}(E_\alpha)$ satisfies
\[ \dimu \pi_{\lambda}(E_\alpha) = (a,\sum_{y \in \cQ(a)_0 \cap C^-_r} \alpha_y) = (a,b). \]
By Theorem $\ref{TheoremRingelGabriel}$ the representation is indecomposable in $\rep(\Gamma_r)$. Moreover we have for each source in $x \in (C_r)_0$ that either $(E_\alpha)_x = 0$ or $\dim_k (E_\alpha)_x = k$ and $|\supp(E_\alpha) \cap n_{C_r}(x)| = r$. Since $E_\alpha$ is indecomposable we conclude that every map $E_\alpha(\delta)$ is injective for each arrow $\delta \in (C_r)_1$. Therefore $E_\alpha \in \Inj$ and $\pi_{\lambda}(E_\alpha) \in \EKP$ by Theorem $\ref{Theorem:INJEKP}$.
\end{proof}

\subsubsection{The case $(r-\frac{1}{r-1})a < b < (\frac{r+\sqrt{r^2-4}}{2})a$}

Let us deal with the last remaining case. 

\begin{Lemma}[compare {\cite[1.1]{Ri6}}]\label{Lemma:ExistenceThinInjectiveArrow}
Let $(u,v) \in \NN^2$ such that $u \leq v \leq (r-1)u + 1$. There exists an indecomposable and thin representation $T_{(u,v)} \in \rep(C_r)$ such that $T_{(u,v)}(\gamma)$ is injective for each $\gamma \in (C_r)_1$ with $\pi(\gamma) = \gamma_1$ 
and $\dimu \pi_\lambda(T_{(u,v)}) = (u,v)$.
\end{Lemma}
\begin{proof} Recall that a representation $M \in \rep(C_r)$ is called thin if $\dim_k M_z \leq 1$ for all $z \in (C_r)_0$.
We consider an unoriented path in $C_r$ that is of the following form
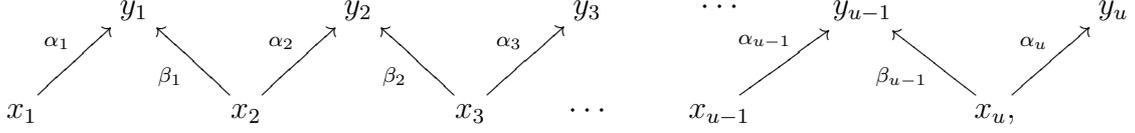
\[ \xymatrix{
& y_1 &  & y_2 && y_3  & \cdots & y_{u-1} && y_u \\
x_1 \ar^{\alpha_1}[ur]&  & x_2 \ar^{\alpha_2}[ur] \ar^{\beta_1}[ul]  & & x_3 \ar^{\alpha_3}[ur] \ar^{\beta_2}[ul]  & \cdots & x_{u-1} \ar^{\alpha_{u-1}}[ur] && x_u, \ar^{\alpha_u}[ur] \ar^{\beta_{u-1}}[ul]
}
\]
such that $\pi(\alpha_i) = \gamma_1$ and $\pi(\beta_i) = \gamma_2$ for all $i$.
Note that 
\[ | \bigcup^{u}_{i=1} n_{C_r}(x_i)| = r-1 + u + (u-1)(r-2) = (r-1)u+1 \geq v \geq u.\]
Hence we find $\{w_1,\ldots,w_v\} \subseteq \bigcup^u_{i=1} n_{C_r}(x_i)$ such that 
$|\{w_1,\ldots,w_v\}| = v$ and $\{y_1,\ldots,y_u\} \subseteq \{w_1,\ldots,w_v\}$. We define a thin representation $T_{(u,v)}$ with ${T_{(u,v)}}_z = k$ if and only if $z \in \{x_1,\ldots,x_u\} \cup \{w_1,\ldots,w_v\}$ and $T_{(u,v)}(x_i \to z) = \id_k$ for all $i \in \{1,\ldots,u\}$ and all $z \in \{w_1,\ldots,w_v\} \cap n_{C_r}(x_i)$.
By construction $T_{(u,v)}$ is a thin and indecomposable. Moreover we have $\{0\} = \ker T_{(u,v)}(\delta)$ for all $\delta \colon x \to y \in (C_r)_1$ such that $\pi(\delta) = \gamma_1$ since $\pi(\delta) = \gamma_1$ implies $\delta = \alpha_i$ for some $i \in \{1,\ldots,u\}$ or $x \not \in \supp(T_{(u,v)})$. By construction we have $\dimu \pi_{\lambda}(T_{(u,v)}) = (u,v)$.
\end{proof}

\noindent We denote by $\Phi := \begin{pmatrix}
r^2 -1  & -r \\
r & -1 
\end{pmatrix}$ the \textsf{Coxeter matrix} of $\Gamma_r$. Recall that $\dimu \tau_{\Gamma_r} M = \Phi \dimu M$ for each regular indecomposable representation $M$.

\begin{Lemma}\label{Lemma:ShiftRepresentation}
Let $(u,v) \in \NN^2$. If $(r-\frac{1}{r-1})u <  v < (\frac{r+\sqrt{r^2-4}}{2})u$, then we find $l \in \NN$ such that $(u_l,v_l)^t := \Phi^l(u,v)^t$ satisfies $u_l < v_l \leq (r-\frac{1}{r-1})u_l$.
\end{Lemma}
\begin{proof} At first note that $q_{\Gamma_r}(u,v) \leq 0$ by Example $\ref{Examples}$. Hence $(u,v)$ is the dimension vector of an indecomposable regular representation and therefore $(u_l,v_l)$, as well. In particular, $v_l < (\frac{r+\sqrt{r^2-4}}{2})u_l$ for all $l \in \NN$. \\
We have $(u_1,v_1)^t = \begin{pmatrix}
r^2-1 & -r\\
r & -1
\end{pmatrix} \begin{pmatrix}
u \\
v
\end{pmatrix} = (r^2u -rv-u,ru-v)^t$.
Since  $ (r-\frac{1}{r-1})u <  v$, we have $(r-1)(ru-v) < u$ and conclude
$u_1-v_1 = r(ru-v)-u - (ru-v) = (r-1)(ru-v) -u < 0$, i.e. $u_1 < v_1$. It follows
$u_1 < v_1 < (\frac{r+\sqrt{r^2-4}}{2})u_1$. If $v_1 \leq (r-\frac{1}{r-1})u_1$, then we are done. Otherwise we have $(r-\frac{1}{r-1}) u_1 < v_1 < (\frac{r+\sqrt{r^2-4}}{2})u_1$ and continue the argument with $(u_1,v_1)$. Since there is $m \in \NN$ such that $(u_m,v_m)^t := \Phi^m (u,v)^t$ satisfies $u_m \geq v_m$ $($see for example \cite[3.1.2]{Wor1}$)$, we conclude that there is $l \in \{1,\ldots,m\}$ such that $u_l < v_l \leq (r - \frac{1}{r-1})u_l$.
\end{proof}

\begin{proposition}\label{Result3}
Let $(a,b) \in \NN^2$ such that $(r - \frac{1}{r-1})a < b < a(\frac{r+\sqrt{r^2-4}}{2})$. There exists an indecomposable representation $F_{(a,b)} \in \rep(C_r)$ such that $F_{(a,b)} \in \Inj$ and $\dimu \pi_\lambda (F_{(a,b)}) = (a,b)$.
\end{proposition}
\begin{proof}
Lemma $\ref{Lemma:ShiftRepresentation}$ provides $l \in \NN$ such that $(u,v)^t := \Phi^l(a,b)^t$ satisfies $u < v \leq (r-\frac{1}{r-1})u$. 
We distinguish the following cases:

\begin{enumerate}
\item If $(r-1)u+1 \leq v \leq (r-\frac{1}{r-1})u$, then Lemma $\ref{Result2}$ yields an indecomposable representation $F_{(u,v)} \in \Inj$ such that $\dimu \pi_{\lambda}(F_{(u,v)}) = (u,v)$. Since $\Inj$ is closed under $\tau^{-1}_{C_r}$ $($see \cite[4.3]{Bi1}$)$, we conclude that $F_{(a,b)} := \tau^{-l}_{C_r} F_{(u,v)} \in \Inj$ is indecomposable with $\dimu \pi_\lambda(F_{(a,b)})^t = \Phi^{-l} (u,v)^t = (a,b)^t$.
\item  If $u < v \leq (r-1)u+1$, then $\ref{Lemma:ExistenceThinInjectiveArrow}$ yields the existence of an indecomposable and thin representation $T_{(u,v)} \in \rep(C_r)$ such that $T_{(u,v)}(\gamma)$ is injective for each $\gamma \in (C_r)_1$ with $\pi(\gamma) = \gamma_1$ 
and $\dimu \pi_\lambda (T_{(u,v)}) = (u,v)$. 
In view of \cite[4.5]{Bi1} we have 
\[0 = \bigoplus_{g \in G} \Hom_{C_r}((X^1)^g,T_{(u,v)}) \cong \Hom(X_{e_1},\pi_\lambda(T_{(u,v)})),\]
where $e_1 \in k^r \setminus \{0\}$ is the first canonical basis vector.
Let $S :=  \pi_\lambda(T_{(u,v)})$ and $\alpha \in k^r \setminus \{0\}$. The assumption 
\[\Hom(\tau_{\Gamma_r} X_\alpha,S) \cong \Hom(X_\alpha,\tau^{-1}_{\Gamma_r} S) \neq 0\]
 yields a non-zero morphism $f \colon \tau_{\Gamma_r} X_\alpha \to S$. By the Euler-Ringel form we have 
 \begin{align*}
 0 &<  r - 2 = 2(r-1)-r = \langle  (1,r-1),(r-1,1) \rangle_{\Gamma_r} \\
 &=  \dim_k \Hom(X_{e_1},\tau_{\Gamma_r} X_\alpha) - \dim_k \Ext^1(X_{e_1},\tau_{\Gamma_r} X_\alpha),
\end{align*}
and conclude $0 \neq \Hom(X_{e_1},\tau_{\Gamma_r} X_\alpha)$. Now \cite[2.1.1, 2.1.4]{Bi1} yield $0 \neq \Hom(X_{e_1},S)$ since $\tau_{\Gamma_r} X_\alpha$ is elementary, a contradiction. We conclude with Theorem $\ref{Theorem:INJEKP}$ that $\tau^{-1}_{\Gamma_r} S \in \EKP$. Since $\tau_{\Gamma_r} \circ \pi_\lambda = \pi_\lambda \circ \tau_{C_r}$, the representation $F_{(a,b)} := \tau^{-l}_{C_r} T_{(u,v)}$ is in $\Inj$  with $\dimu \pi_\lambda (F_{(a,b)}) = (a,b)$. 
\end{enumerate} 
\end{proof}

\begin{Remark}
The arguments in $(b)$ have already been used \cite[4.8]{Wor3}. The author calls these representations \textsf{locally injective}. For the sake of completeness, we decided to give all the details. 
\end{Remark}

\section{The main results}

\noindent Let us collect the main results of this article.

\begin{Theorem}\label{Theorem:ExistenceKronecker}
Let $r \geq 2$ and $(a,b) \in \NN^2_0$ such that $q_{\Gamma_r}(a,b) \leq 1$. If $b - a \geq r - 1$ or $(a,b) = (0,1)$, then there exists an indecomposable representation $V_{(a,b)}$ such that $V_{(a,b)}$ has the equal kernels property and $\dimu V_{(a,b)} = (a,b)$. Hence
\[ \dimu \EKP =  \{ (a,b) \in \NN_0 \times \NN_0  \mid q_{\Gamma_r}(a,b) \leq 1, b-a \geq r-1\} \cup \{(0,1)\} = \cD.\]
\end{Theorem}
\begin{proof}
Let $(a,b) \in \cD$. As mentioned before, we only have to deal with the case $q_{\Gamma_r}(a,b) < 1$. For $r = 2$ this can not happen since $b \geq a + 1$. For $r \geq 3$ the statement follows from $\ref{Result1}$, $\ref{Result2}$ and $\ref{Result3}$. 
\end{proof}

\begin{corollary}\label{Corollary:ExistenceAllp}
Let $(c,d) \in \NN^2_0$. There exists an indecomposable representation of constant Jordan type $[1]^c [2]^d$ if and only if  $(c,d) \in \IJT \cup \{(1,0)\}$. 
\end{corollary}
\begin{proof} We conclude with Theorem $\ref{Theorem:ExistenceKronecker}$ and Lemma $\ref{Lemma:Bijection}$
\[\Ad = \Xi^{-1}(\dimu \EKP) = \Xi^{-1}(\cD) = \IJT \cup \{(1,0)\}.\]
Now apply Theorem $\ref{Theorem:ConstantJordanTypeRestrictions}$.
\end{proof}

\noindent Since an indecomposable module $M \in \modd kE_r$ has of Loewy length $1$ if and only if $M$ has constant Jordan type $[1]^1 [2]^0$, we conclude with Proposition $\ref{Proposition:WorchEmbedding}$ and the fact that the Jordan type does not change under $D_{\Gamma_r}$:

\begin{Theorem}\label{Theorem:EquivalenceAllp}
Let $\Char(k) = p > 0$, $r \geq 2$ and $(c,d) \in \NN^2_0$.
 The following statements are equivalent:\begin{enumerate}
\item  There exists an indecomposable $kE_r$-module of constant Jordan type $[1]^c [2]^d$ and Loewy length $2$.
\item There exists an indecomposable $kE_r$-module with the equal images property of constant Jordan type $[1]^c [2]^d$ and Loewy length $2$.
\item There exists an indecomposable $kE_r$-module with the equal kernels property of constant Jordan type $[1]^c [2]^d$ and Loewy length $2$.
\item $(c,d) \in \IJT$.
\end{enumerate}
\end{Theorem}

\begin{corollary}
Let $\Char(k) = p > 2$, $r \geq 2$. For each element $(c,d) \in \NN_0 \times \NN$ the following statements are equivalent:
\begin{enumerate}
\item  There exists an indecomposable $kE_r$-module of constant Jordan type $[1]^c [2]^d$.
\item There exists an indecomposable $kE_r$-module with the equal images property of constant Jordan type $[1]^c [2]^d$.
\item There exists an indecomposable $kE_r$-module with the equal kernels property of constant Jordan type $[1]^c [2]^d$.
\item $(c,d) \in \IJT$.
\end{enumerate}
\end{corollary}
\begin{proof}
 Let $M \in \modd kE_r$ be an indecomposable  representation of constant Jordan type $[1]^c [2]^d$ with $(c,d) \in \NN_0 \times \NN$. 
Since $2 < p$ we have $\Rad^2(kE_r) = (\{x^2_\alpha \mid \alpha \in k^r\setminus \{0\})$ by \cite[1.17.1]{Be1}. Hence $(\{x^2_\alpha \mid \alpha \in k^r\setminus \{0\}).M = 0$ and $M$ has Loewy length $\leq 2$. Since $d \neq 0$, we conclude that $M$ has Loewy length $2$. Now apply Theorem $\ref{Theorem:EquivalenceAllp}$
\end{proof}

\begin{Remark}
Consider $p = 2$ and $r = 2$. The regular module $kE_2$ has Loewy length $3$ and constant Jordan type $[1]^0 [2]^2$. We have $q_{\Gamma_2}(2,2) \leq 1$ but $(0,2) \not \in \IJT$.
\end{Remark}

\begin{corollary}
Let $r \geq 3$ and $(c,d) \in \NN^2_0$ such that $r-1 \leq c \leq (r-2)d$. Then the $kE_r$-modules with the equal kernels property of constant Jordan type $[1]^{cn} [2]^{dn}$, $n \in \NN$ have wild representation type.
\end{corollary}
\begin{proof}
Since $c \leq (r-2)d$ we have $d \leq d + c \leq (r-1)d$. We set $a := d$ and $b := d+c$. In view of $\ref{Result1}$, we find a regular brick $F_{(a,b)}$ in $\EKP$ with dimension vector $(a,b)$. 
Now \cite[3.1.1]{Bi1} implies that the full subcategory $\cE(\{F_{(a,b)}\}) \subseteq \rep(\Gamma_r)$ of representations that have a $\{F_{(a,b)}\}$-filtration is of wild representation type. For $U \in \cE(\{F_{(a,b)}\})$ indecomposable we find $n \in \NN_0$ such that $\dimu U = (na,nb)$. Since $\EKP$ is closed under extensions, we have $U \in \cE(\{F_{(a,b)}\}) \subseteq \EKP$ and $U$ is of constant Jordan type $[1]^{nc} [2]^{nd}$. Now use $\ref{Proposition:WorchEmbedding}$ to conclude the result for $\modd kE_r$.
\end{proof}

\section{Gradable Representations}

\noindent We show that an indecomposable representation $M \in \rep(C_r)$ such that $N := \pi_\lambda(M) \in \EKP$ satisfies $\dim_k N_2 \geq (r-1) \dim_k N_1 + 1$. 

\begin{Lemma}\label{Lemma:SourceRegularTree}
Let $T \subseteq C_r$ be source-regular with $n \in \NN$ vertices. We denote by $T^+$  the sources of $T$ and by $T^-$ the sinks. Assume $\bar{} \colon T_0 \to \NN$ is a map such that for all $a \in T^+$ we have $\overline{a} \leq \overline{b}$ for all $b \in n_T(a)$. Let $m := \max \{ \overline{a}  \mid a \in T^+\}$,  then 
\[m + (r-1) \sum_{a \in T^+} \overline{a} \leq \sum_{b \in T^-} \overline{b} .\]
\end{Lemma}
\begin{proof}
We prove the statement by induction on $n = |T^+|$. We write $T^+ = \{x_1,\ldots,x_n\}$ and consider the case $n = 1$, then $\overline{x}_1 = m$ and $T_0 = \{x_1\} \cup n_T(x_1)$. It follows 
\[  \sum_{b \in T^-} \overline{b} = \sum_{b \in n_T(x_1)} \overline{b} \geq r\overline{x}_1 = m + (r-1)\overline{x}_1 = m + (r-1) \sum_{a \in T_+} \overline{a}. \]
Now let $n \geq 2$ and $x \in T^+$ such that $\overline{x} = \min\{ \overline{a} \mid a \in T^+\}$. We have $n_T(x) = \{b_1,\ldots,b_{r}\}$ and $|n_T(x)| = r $. For each $i \in \{1,\ldots,r\}$ we denote with $T(i)$ the maximal full subtree of $T$ such that $b_i \in T(i)_0$ and $x \not\in T(i)_0$. 
Since $n \geq 2$, there is $i \in \{1,\ldots,r\}$ such that $T(i)_0 \cap T^+ \neq \emptyset$, i.e. $T(i)_0 \neq \{b_i\}$. Without loss of generality we can assume that $1 \leq l \leq r$ is maximal such that $T(i)_0 \cap T^+ \neq \emptyset$ for all $1 \leq i \leq l$, i.e. $T(i)$ is source-regular with $< n$ sources for all $1 \leq i \leq l$. We  define $m_i := \max \{ \overline{a} \mid a \in T(i)_0 \cap T^+\}$ for all $i \in \{1,\ldots,l\}$ and assume without loss of generality that $m = m_l$. We get with the inductive hypothesis
\begin{align*}
\sum_{b \in T^-} \overline{b} &= \sum^{l}_{i=1} \sum_{b \in T(i)_0 \cap T^-} \overline{b} + \sum^{r}_{l+1} \overline{b}_i\geq (r-1) (\sum^l_{i=1} \sum_{a \in T(i)_0 \cap T^+}  \overline{a}) + \sum^l_{i=1} m_i + \sum^{r}_{l+1} \overline{b}_i \\
&= (r-1) \sum_{a \in T^+} \overline{a} - (r-1)\overline{x} + \sum^l_{i=1} m_i + \sum^r_{i=l+1} \overline{b}_i \\
&= (r-1) \sum_{a \in T^+} \overline{a} + m + \sum^{l-1}_{i=1} m_i + \sum^r_{i=l+1}\overline{b}_i - (r-1)\overline{x}\\ 
&\geq (r-1) \sum_{a \in T^+} \overline{a} + m + \sum^{l-1}_{i=1} \overline{x} + \sum^r_{i=l+1}\overline{x} - (r-1)\overline{x}\\ 
&\geq (r-1) \sum_{a \in T^+} \overline{a} + m.
\end{align*}
\end{proof}

\begin{Lemma}
Let $M \in \Inj \subseteq \rep(C_r)$ be indecomposable. Assume that $(a,b) = \dimu \pi_\lambda(M)$ and let $x_1
,\ldots,x_n$ such that $\{x_1,\ldots,x_n\} = \supp(M) \cap (C_r)^+$ and $|\{x_1,\ldots,x_n\}| = n$. Moreover, let $m := \max \{ \dim_k M_{x_i} \mid i \in \{1,\ldots,n\} \}$. Then the full subquiver with vertex set $\supp(M)$ is source-regular with $n$ vertices and 
\[ b \geq (r-1)a+ m  \geq (r-1)a+1.\]
\end{Lemma}

\begin{proof}
Let $i \in \{1,\ldots,n\}$ and $y \in n_{C_r}(x_i) = x_i^+$, then $y \in \supp(M)$ since the $k$-linear map $M_{x_i} \to M_y$ is injective. We conclude 
\[ \{ y \in (C_r)_0 \mid \exists i \in \{1,\ldots,n\} \colon y \in n_{C_r}(x_i) \} \cup \{x_1,\ldots,x_n\}  \subseteq \supp(M).\]
 Since $\{x_1,\ldots,x_n\} = \supp(M) \cap (C_r)^+$, we conclude 
\[ \{ y \in (C_r)_0 \mid \exists i \in \{1,\ldots,n\} \colon y \in n_{C_r}(x_i)  \} \cup \{x_1,\ldots,x_n\} = \supp(M).\]
Hence $\supp(M)$ induces a source-regular tree $T$ and for $x \in T^+$ and all $y \in n_T(x)$ we have $\dim_k M_x \leq \dim_k M_y$. Now Lemma $\ref{Lemma:SourceRegularTree}$ implies 
\[m + (r-1)a = m + (r-1) \sum_{x \in T^+} \dim_k M_a \leq \sum_{y \in T^-} \dim_k M_y  = b.\]
\end{proof}

\section*{Acknowledgement}

\noindent I would like to thank Rolf Farnsteiner
 for his detailed comments, on an earlier version of this article, that helped to improve the exposition of this article.



\begin{bibdiv}
\begin{biblist}
\addcontentsline{toc}{chapter}{\textbf{Bibliography}}

\bib{Assem1}{book}{
title={Elements of the Representation Theory of Associative Algebras, I}
subtitle={Techniques of Representation Theory}
series={London Mathematical Society Student Texts}
author={I. Assem},
author={D. Simson},
author={A. Skowro\'nski},
publisher={Cambridge University Press},
date={2006},
address={Cambridge}
}

\bib{Be1}{book}{
title={Representations of elementary abelian p-groups and vector bundles},
author={D. Benson},
publisher={Cambridge University Press},
series={Cambridge Tracts in Mathematics},
volume={208},
date={2016},
}

\bib{Bi1}{unpublished}{
title={Representations of constant socle rank for the Kronecker algebra},
author={D. Bissinger},
date={2016},
status={Preprint, arXiv:1610.01377v1}
}

\bib{Bi2}{article}{
title={Representations of Regular Trees and Invariants of AR-Components for Generalized Kronecker Quivers},
author={D. Bissinger},
date={2018},
journal={Algebras and Representation Theory},
volume={21},
number={2},
pages={331-358}
}

\bib{Bi3}{webpage}{
author ={D. Bissinger},
title={Representations of regular trees and wild subcategories for generalized Kronecker quivers. PhD-Thesis},
url={http://macau.uni-kiel.de/receive/dissertation_diss_00013419},
date={2017}
}

\bib{BoChen1}{article}{
title={Dimension vectors in regular components over wild
Kronecker quivers},
journal={Bulletin des Sciences Math\'{e}matiques},
volume={137},
pages={730-745},
author={B. Chen},
date={2013}
}

\bib{CFP1}{article}{
title={Modules of Constant Jordan Type},
author={Carlson, J. F.},
author={Friedlander, E. M.},
author={Pevtsova, J.},
date={2008},
journal={J. Reine Angew. Math.},
volume={614},
pages={191-234}
}

\bib{CB1}{webpage}{
title={Geometry of representations of algebras},
author={W. Crawley-Boevey},
url={https://www.math.uni-bielefeld.de/~wcrawley/geomreps.pdf}
date={1993}
}


\bib{Far2}{unpublished}{
author ={R. Farnsteiner},
title={Lectures Notes: Nilpotent Operators, Categories of Modules, and Auslander-Reiten Theory},
note={http://www.math.uni-kiel.de/algebra/de/farnsteiner/material/Shanghai-2012-Lectures.pdf},
date={2012}
}

\bib{Gab2}{article}{
title={Covering spaces in representation theory},
author={K. Bongartz},
author={P. Gabriel},
journal={Invent. Math.},	
year={1981/82},
pages = {331-378},
volume = {65},
}

\bib{Gab3}{book}{
title={The universal cover of a representation finite algebra},
series={Representations of Algebras, Lecture Notes in Mathematics},
volume={903},
pages={68-105},
date={1981},
author={P. Gabriel},
publisher={Springer}
address={Berlin}
}

\bib{Kac1}{article}{
title={Infinite roots systems, representations of graphs and invariant theory},
author={V. G. Kac},
journal={Invent. Math.},
volume={56},
date={1980},
pages={57-92}
}

\bib{Kac3}{article}{
title={Infinite root systems, representations of graphs and invariant theory, II},
author={V. G. Kac},
journal={J. Algebra},
volume={78},
date={1982},
number={1}
pages={141–-162}
}

\bib{Kac2}{book}{
title={Root systems, representations of quivers and invariant theory},
series={Invariant theory(Montecatini), Lecture Notes in Mathematics},
volume={996},
pages={74-108},
date={1982},
publisher={Springer}
author={V. G. Kac},
address={Berlin}
}

\bib{Ker3}{article}{
title={Representations of Wild Quivers},
journal={Representation theory of algebras and related topics, CMS Conf. Proc.},
volume={19},
date={1996},
pages={65-107}, 
author={O. Kerner},
}

\bib{KerLuk1}{article}{
title={Elementary modules},
author={O. Kerner},
author={F. Lukas},
journal={Math. Z.},
volume={223},
pages={421-434},
date={1996}
}

\bib{Ri7}{webpage}{
title={Covering Theory},
author={C.M. Ringel},
url={https://www.math.uni-bielefeld.de/~ringel/lectures/izmir/izmir-6.pdf}
}

\bib{Ri6}{article}{
title={Indecomposable representations of the Kronecker quivers},
author={C.M. Ringel},
journal={Proc. Amer. Math. Soc.},
volume={141},
date={2013},
number={1},
pages={115--121}
}

\bib{Ri10}{article}{
title={The elementary 3-Kronecker modules}
author={C.M. Ringel},
status={Preprint, arXiv:1612.09141},
date={2016}
}

\bib{Assem2}{book}{
title={Elements of the Representation Theory of Associative Algebras, II},
subtitle={Tubes and Concealed Algebras of Euclidean type},
series={London Mathematical Society Student Texts},
author={D. Simson},
author={A. Skowro\'nski},
publisher={Cambridge University Press},
date={2007},
address={Cambridge}
}

\bib{Wor1}{article}{
title={Categories of modules for elementary abelian p-groups and generalized Beilinson algebras},
author={J. Worch},
journal={J. London Math. Soc.},
volume={88},
date={2013},
pages={649-668}
}

\bib{Wor3}{webpage}{
author ={J. Worch},
title={Module categories and Auslander-Reiten theory for generalized Beilinson algebras. PhD-Thesis},
url={http://macau.uni-kiel.de/receive/dissertation_diss_00013419},
date={2013}
}


\end{biblist}
\end{bibdiv}
\end{document}